\documentclass[aos,preprint]{imsart}

\RequirePackage{amsthm,amsmath,amsfonts,amssymb}
\RequirePackage[numbers]{natbib}
\RequirePackage[ruled,vlined,linesnumbered]{algorithm2e}
\RequirePackage[colorlinks,citecolor=blue,urlcolor=blue]{hyperref}
\RequirePackage{graphicx}
\usepackage{tabularx}

\startlocaldefs

\newtheorem{theorem}{Theorem}[section]

\newtheorem{definition}[theorem]{Definition}

\newtheorem{corollary}[theorem]{Corollary}
\newtheorem{proposition}[theorem]{Proposition}
\newtheorem*{claim*}{Claim}
\theoremstyle{AppDefinition}

\theoremstyle{AppClaim}

\newtheorem{remark}[theorem]{Remark}

\newtheorem{example}[theorem]{Example}
\newtheorem*{example*}{Example}
\newtheorem{problem}[theorem]{Problem}

\def\R{{\mathbb R}}
\newcommand{\Trop}{\mathrm{Trop}}
\newcommand{\one}{{\bf{1}}}
\DeclareMathOperator*{\argmax}{\arg\!\max}

\def\beginmat{ \left( \begin{array} }
\def\endmat{ \end{array} \right) }

  \endlocaldefs

\begin{document}

\begin{frontmatter}
\title{Tropical Support Vector Machine  and its Applications to Phylogenomics}
\runtitle{Tropical Support Vector Machine}

\begin{aug}
  \author[A]{\fnms{Xiaoxian} \snm{Tang}\ead[label=e1]{xiaoxian@buaa.edu.cn}},
\author[B]{\fnms{Houjie} \snm{Wang}\ead[label=e2]{wanghoujie@tamu.edu}}
\and
\author[C]{\fnms{Ruriko} \snm{Yoshida}\ead[label=e3]{ryoshida@nps.edu}}
  \address[A]{School of Mathematical Sciences, Beihang University,
    \printead{e1}}

\address[B]{Department of Statistics, Texas A\&M University,
  \printead{e2}}
\address[C]{Department of Operations Research, Naval Postgraduate School,
  \printead{e3}}
\end{aug}

\begin{abstract}
Most data in genome-wide phylogenetic analysis (phylogenomics) is essentially multidimensional, posing a major challenge to human comprehension and computational analysis. Also, we can not directly apply statistical learning models in data science to a set of phylogenetic trees since the space of phylogenetic trees is not Euclidean.  In fact, the space of phylogenetic trees is a tropical Grassmannian in terms of max-plus algebra. Therefore, to classify  multi-locus data sets for phylogenetic analysis, we propose tropical support vector machines (SVMs). Like classical SVMs, a tropical SVM is a discriminative classifier defined by the tropical hyperplane which maximizes the minimum tropical distance from data points to itself in order to separate these data points into sectors (half-spaces) in the tropical projective torus. Both hard margin tropical SVMs and soft margin tropical SVMs  can be formulated as linear programming problems. We focus on classifying two categories of data, and we study a simpler case by assuming the data points from the same category ideally stay in the same sector of a tropical separating hyperplane.  
For hard margin tropical SVMs, we prove the necessary and sufficient conditions for two categories of data points to be separated, and
 we show an explicit formula for the optimal value of the feasible linear programming problem.  For soft margin tropical SVMs, we develop novel methods to compute an optimal tropical separating  hyperplane. Computational experiments show our methods work well.  We end this paper with open problems.
\end{abstract}



\end{frontmatter}

\section{Introduction}\label{sec:intro}
Multi-locus data sets in phylogenomics are essentially multidimensional.  Thus, we wish to apply tools from data science to analyse how phylogenetic trees of different genes (loci) are distributed over the space of phylogenetic trees.  For example, in many situations in systematic biology, we wish to classify multi-locus phylogenetic trees over the space of phylogenetic trees.  
  In order to apply tools from Phylogenomics to multi-locus data sets, systematists exclusively select alignments of protein or DNA sequences whose evolutionary histories are congruent to these  respective of their species. 
  In order to see how alignments with such evolutionary events differ from others and to extract important information from these alignments, systematists compare sets of multiple phylogenetic trees generated from different genomic regions to assess concordance or discordance among these trees across genes \cite{ane2007}.
  
  This problem appears not only in analysis on multi-locus phylogenetic data but also assessing convergence of Markov Chain Monte Carlo (MCMC) analyses for the Bayesian inference on phylogenetic tree reconstruction.  When we conduct an MCMC analysis on phylogenetic tree reconstruction, we run independent multiple chains on the same data and we have to ensure they converge to the same posterior distribution of trees.  Often this process is done by comparing summary statistics computed from  sampled trees, however, naturally computing a summary statistic from a sample loses information about the sample \cite{Wilgenbusch}.  
  
  In a Euclidean space, we apply a supervised learning method to classify data points into different categories.  
  A support vector machine (SVM) is one of the most popular supervised learning models for  classification.   In a Euclidean space, an SVM is a linear classifier, a hyperplane which  separates data into half-spaces and maximizes the minimum distances from data points to the hyperplane.  A space of all possible phylogenetic trees with the same set of labels on their leaves is unfortunately not Euclidean.  In addition, this space is a union of lower dimensional polyhedral cones inside of a Euclidean space \cite{SS}.  Therefore we cannot directly apply a classical SVM to a set of phylogenetic trees.  
  
  In 2004, Speyer and Sturmfels showed a space of phylogenetic trees with a given set of labels on their leaves is a tropical Grassmanian \cite{SS}, which is a tropicalization of a linear space defined by a set of linear equations \cite{YZZ} with the max-plus algebra.  Therefore, in this paper, we propose applying a tropical SVM to the data sets of phylogenetic trees.  
  

  Similar to a classical SVM over a Euclidean space, a   {\em hard margin} tropical   SVM introduced by \cite{Gartner} is a tropical hyperplane which  maximizes the {\em margin}, the minimum tropical distance from data points to the tropical hyperplane (which is $z$ in Figure \ref{fig:tropHardMargin}), to separate these data points into {\em open sectors} over a tropical projective torus.  
   Similar to the classical hard margin SVMs, hard margin tropical SVMs assume that there is a tropical hyperplane which separates all points from different categories into each open sector (see the left figure in Figure \ref{fig:tropHardMargin}).  
  \begin{figure}
      \centering
      \includegraphics[width=0.4\textwidth]{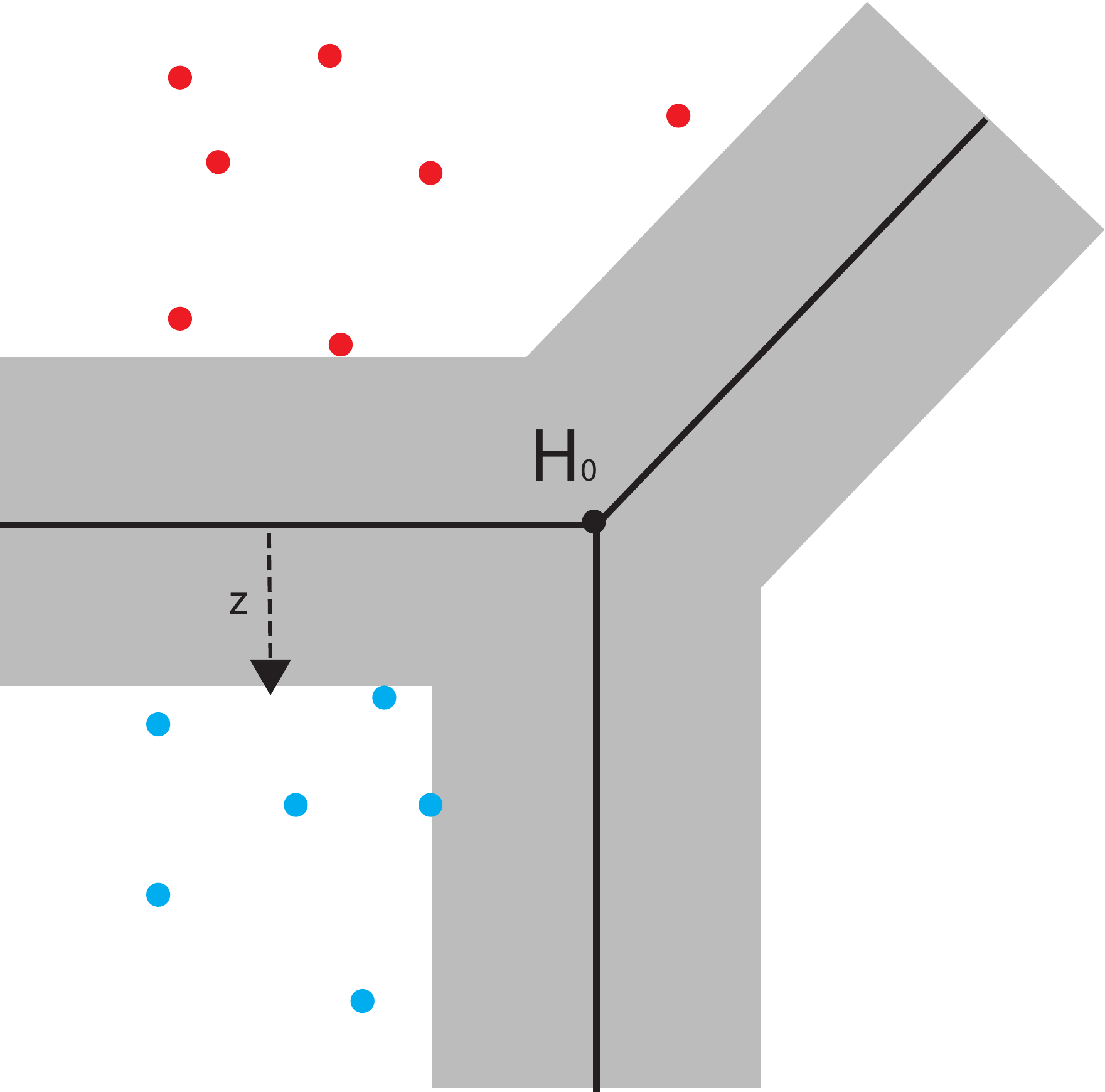}
      \includegraphics[width=0.4\textwidth]{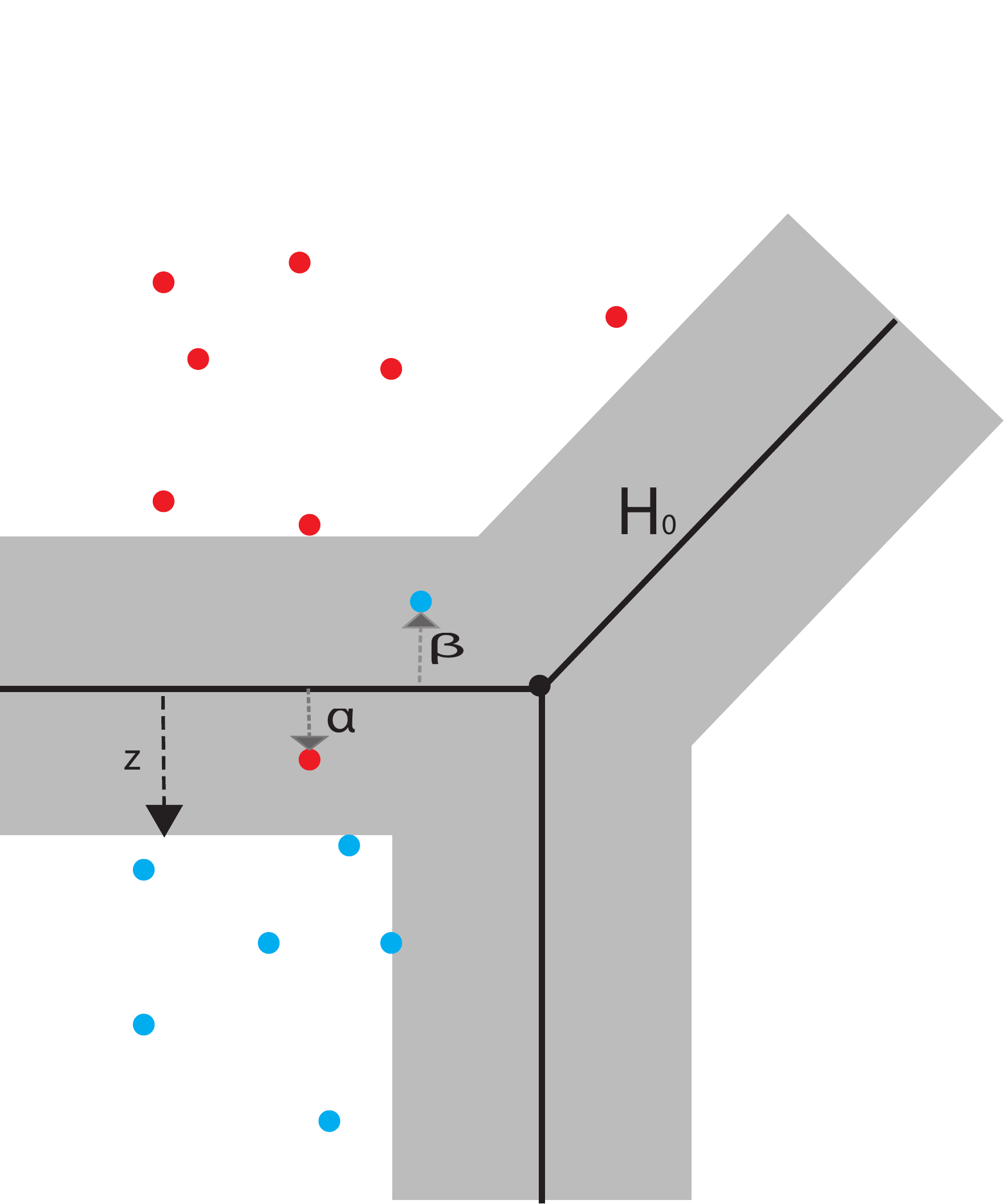}
      \caption{Picture of a hard margin tropical SVM (LEFT) and a soft margin tropical SVM (RIGHT) with two categories. A hard margin tropical SVM assumes that data points are separated by a tropical hyperplane and a hard margin tropical SVM maximizes the margin $z$ which is the width of the grey area from the tropical hyperplane in this figure. A soft margin tropical SVM  maximizes a margin and minimizes the sum of $\alpha$ and $\beta$ at the same time.}
      \label{fig:tropHardMargin}
  \end{figure}
 By the ideas proposed in  \cite{Gartner},  we formulate hard margin tropical SVMs as linear programming problems (see \eqref{equation:251}--\eqref{equation:254}). Since a hard margin tropical SVM assumes all data points from different categories are clearly separated by a tropical hyperplane, we have to check which points are in which sector. To do so,  
  we have to go through possibly exponentially many linear programming problems in terms of the dimension of the data and the sample size of the input data set.  
  In practice, many of these linear programming problems might be infeasible. 
Here, we study a special case when the data points from the same category stay in the same sector of a  separating tropical hyperplane.   
For this simpler case, we show the necessary and sufficient conditions for each linear programming problem to be feasible and  an explicit formula for the optimal value of a feasible linear programming problem (see Theorems \ref{thm:16}-\ref{theorem76}). 
  
  As a classical SVM, the assumption of a hard margin tropical SVM, such that all data points from different categories are clearly separated, is not realistic in general.  It  rarely happens that all data points from different categories are clearly separated by a tropical hyperplane.  In a Euclidean space, we use soft margin SVMs if some data points from different categories are overlapped.  In this paper, we introduce soft margin tropical SVMs and show the soft margin tropical SVMs can be formulated as linear programming problems by adding slacker variables into the hard margin topical SVMs (see \eqref{equation:421}--\eqref{equation:424}).   We show these linear programming problems are feasible (see Proposition \ref{pp:subfeasible}), and for proper scalar constants, the objective functions of these linear programming problems have bounded optimal values (see Theorem \ref{thm:boundedsoft}). 
  
  Based on our theorems, we develop algorithms  to compute a tropical SVM implemented in {\tt R} (see Algorithms \ref{alg:16}--\ref{alg:14}).   Finally we apply our methods to simulated data generated by the multispecies coalescent processes \cite{coalescent}.  Computational results show our methods are efficient in practice and the accuracy rates of soft margin tropical SVMs are much higher than those of the classical SVMs implemented in the {\tt R} package {\tt e1071} \cite{classicalSVM} for the simulated data with a small ratio of the species depth to the population size (see Figure \ref{fig:accuracy}). 

This paper is organized as follows: In Section \ref{trop:basics}, we remind readers basics in tropical geometry with max-plus algebra.  In Section \ref{subsec:moduli}, we discuss the space of phylogenetic trees with a fixed set of labels on leaves as a tropical Grassmannian with max-plus algebra.  In Section \ref{sec:hardmargin}, we formulate a hard margin tropical SVM as an optimal solution of a linear programming problem. When the data points from the same category stay in the same open sector of a separating tropical hyperplane, we discuss the necessary and sufficient conditions for data points from different categories to be separated via a tropical hard SVM.  In addition, we show the explicit formulae for the optimal values of  feasible linear programing problems.  
In Section \ref{sec:softmargin}, we formulate a soft margin tropical SVM  as  linear programming problems.  Then we prove properties of soft margin tropical SVMs.
In Section \ref{sec:algorithm}, we develop algorithms based on theorems in Section \ref{sec:softmargin}, and we apply them to simulated data generated under the multispecies coalescent processes in Section \ref{experiment}.  Finally in Section \ref{sec:discussion}, we conclude our results and propose open problems. 
The proofs of Theorems \ref{thm:16}--\ref{theorem76} are presented in Appendix \ref{appA}.
 Our software implemented in {\tt R} and simulated data can be downloaded at \url{https://github.com/HoujieWang/Tropical-SVM}, see Appendix \ref{appB}.

\section{Tropical Basics}\label{trop:basics}
Here we review some basics of tropical arithmetic and geometry, as well as setting up the notation through this paper.  For more details, see \cite{maclagan2015introduction} or \cite{J}.

\begin{definition}[Tropical Arithmetic Operations]
  Throughout this paper we will perform arithmetic in the
  max-plus tropical semiring $(\,\mathbb{R} \cup \{-\infty\},\boxplus,\odot)\,$.
  In this tropical semiring,  the basic tropical
  arithmetic operations of addition and multiplication are defined as:
    $$a \boxplus b := \max\{a, b\}, ~~~~ a \odot b := a + b ~~~~\mbox{  where } a, b \in \mathbb{R}\cup\{-\infty\}.$$
    Note that $-\infty$ is the identity element under addition and 0 is the identity element under multiplication.
  \end{definition}
  
  \begin{definition}[Tropical Scalar Multiplication and Vector Addition]
  For any scalars $a,b \in \mathbb{R}\cup \{-\infty\}$ and for any vectors $v = (v_1,
                                                                                 \ldots ,v_e), w= (w_1, \ldots , w_e) \in (\mathbb{R}\cup-\{\infty\})^e$, we 
  define tropical scalar multiplication and tropical vector addition as follows:
    $$a \odot v := (a + v_1,  \ldots ,a + v_e),$$
    $$a \odot v \boxplus b \odot w := (\max\{a+v_1,b+w_1\}, \ldots, \max\{a+v_e,b+w_e\}).$$
    \end{definition}
  
  Throughout this paper we consider the \emph{tropical projective torus}, that is, the projective space
   $\mathbb R^e \!/\mathbb R {\bf 1}$, where ${\bf 1}:=(1, 1, \ldots , 1)$, the all-one vector.  In $\mathbb R^e \!/\mathbb R {\bf 1}$, any
   point $(v_1, \ldots, v_e)$ is equivalent to $(v_1+a, \ldots, v_e+a)$ for any scalar $a\in \mathbb R$. 
  
  \begin{example}
Consider $\mathbb R^e \!/\mathbb R {\bf 1}$.  Then let
\[
v = (1, 2, 3), \, w = (1, 1, 1).
\]
Also let $a = -1, \, b = 3$.  Then we have
\[
a \odot v \boxplus b \odot w = (\max(-1 + 1, 3 + 1), \max(-1 + 2, 3 + 1), \max(-1 + 3, 3 + 1)) = (4, 4, 4) = (0, 0, 0).
\]
  \end{example}
  
  \begin{definition}[Generalized Hilbert Projective Metric]
  For any two points $v, \, w \in \mathbb R^e \!/\mathbb R {\bf
    1}$,  the {\em tropical
      distance} $d_{\rm tr}(v,w)$ between $v$ and $w$ is defined as:
    \begin{equation*}
  \label{eq:tropmetric} d_{\rm tr}(v,w) \,\, = \,\,
  \max_{i,j} \bigl\{\, |v_i - w_i  - v_j + w_j| \,\,:\,\, 1 \leq i < j \leq e \,\bigr\} = \max_{i} \bigl\{ v_i - w_i \bigr\} - \min_{i} \bigl\{ v_i - w_i \bigr\},
  \end{equation*}
  where $v = (v_1, \ldots , v_e)$ and $w= (w_1, \ldots , w_e)$. This
  distance measure is a metric in $\mathbb R^e \!/\mathbb R {\bf
    1}$.
  \end{definition}
  
  \begin{example}
  Suppose $u_1, \, u_2 \in \mathbb R^3 \!/\mathbb R {\bf 1}$ such that
  \[
  u_1 = (0, 0, 0), \, u_2 = (0, 3, 1).
  \]
  Then the tropical distance between $u_1, \, u_2$ is
  \[
  d_{\rm tr}(u_1, u_2) = \max(0, -3, -1) - \min(0, -3, -1) = 0 - (-3) = 3. 
  \]
  \end{example}
  
  \begin{definition}[Tropical Convex Hull]
  The {\em tropical convex hull} or {\em tropical polytope} of a given finite subset $V = \{v_1, \ldots , v_s\}\subset \mathbb R^e \!/\mathbb R {\bf
    1}$ is the smallest tropically-convex subset containing $V \subset \mathbb R^e \!/\mathbb R {\bf
    1}$: it is written as the set of all tropical linear combinations of $V$ such that:
    $$
    \mathrm{tconv}(V) = \{a_1 \odot v_1 \boxplus a_2 \odot v_2 \boxplus \cdots \boxplus a_s \odot v_s \mid v_1,\ldots,v_s \in V \mbox{ and } a_1,\ldots,a_s \in \R \}.
  $$
    A {\em tropical line segment} between two points $v_1, \, v_2$ is the tropical convex hull of $\{v_1, \, v_2\}$.  
  \end{definition}
  
  \begin{example}
  Suppose we have a set $ V = \left\{v_1, \, v_2, \, v_3\right\} \subset \mathbb R^3 \!/\mathbb R {\bf 1}$ where
  \[
  v_1 = (0, 0, 0), \, v_2 = (0, 3, 1), \, v_3 = (0, 2, 5).
  \]
  Then we have the tropical convex hull $\mathrm{tconv}(V)$ of $V$ is shown in Figure \ref{fig:tropPoly}.
  \begin{figure}
      \centering
      \includegraphics[width=0.5\textwidth]{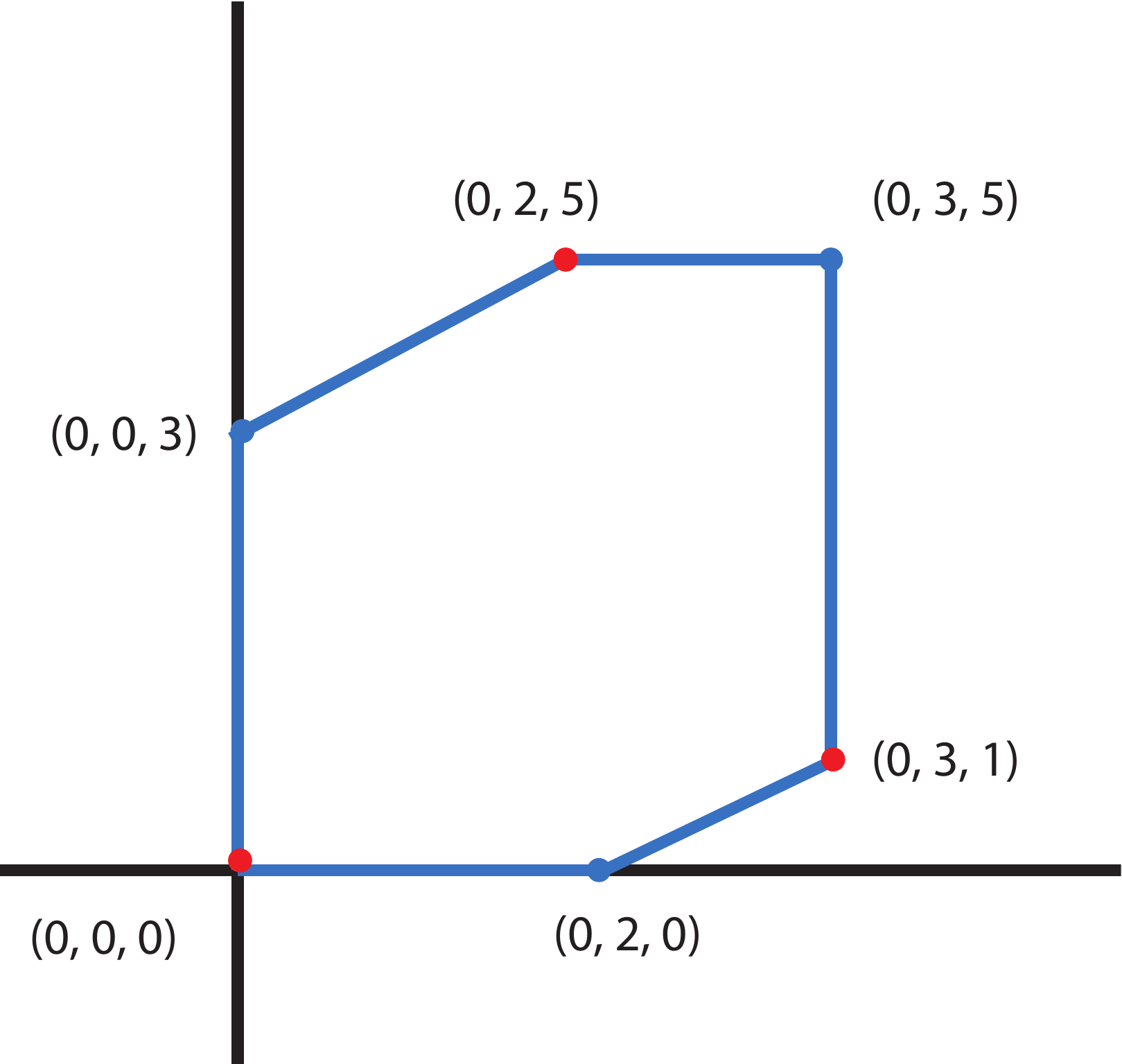}
      \caption{Tropical polytope of three points $(0, 0, 0), \, (0, 3, 1), \, (0, 2, 5)$ in $\mathbb R^3 \!/\mathbb R {\bf 1}$.}
      \label{fig:tropPoly}
  \end{figure}
  \end{example}

\section{Phylogenetic Trees}
  \label{subsec:moduli}
  A phylogenetic tree is a tree representation of evolutionary relationship between taxa.  More formally a phylogenetic tree is a weighted tree with unlabeled internal nodes and labeled leaves.  Weights on edges in a phylogenetic tree represent evolutionary time multiplied by an evolutionary rate. For more details on evolutionary models on phylogenetic trees, see \cite{Steel2003}.  A phylogenetic tree can be rooted or unrooted.  In this paper we focus on rooted phylogenetic trees.  Let $N \in \mathbb{N}$ be the number of leaves and $[N] := \{1, \ldots , N\}$ be the set of labels for leaves. 
  \begin{remark}
  There exist 
  \[
    (2N - 3)!! = (2N - 3)\cdot (2N - 5) \cdots 3 \cdot 1 
    \]
  many binary rooted phylogenetic tree topologies.  
  \end{remark}
    \begin{example}
  Suppose $N = 4$.  Then there are $15$ many different tree topologies for rooted phylogenetic trees.  
  \end{example}
  
  If a total of weights of all edges in a path from the root to each leaf $i \in [N]$ in a rooted phylogenetic tree $T$ is the same for all leaves $i \in [N]$, then we call a phylogenetic tree $T$ {\em equidistant tree}.  Through this paper we focus on equidistant trees with leaves with labels $[N]$. The {\em height} of an equidistant tree is the total weight of all edges in a path from the root to each leaf in the tree.  Through the manuscript we assume that all equidistant trees have the same height.  In phylogenetics this assumption is fairly mild since the multispecies coalescent model assumes that all gene trees have the same height.
  
  \begin{example}
  Suppose $N = 4$.   Rooted phylogenetic trees shown in Figure \ref{fig:equidistant} are equidistant trees with their height equal to $1$.
    \begin{figure}
      \centering
      \includegraphics[width=0.5\textwidth]{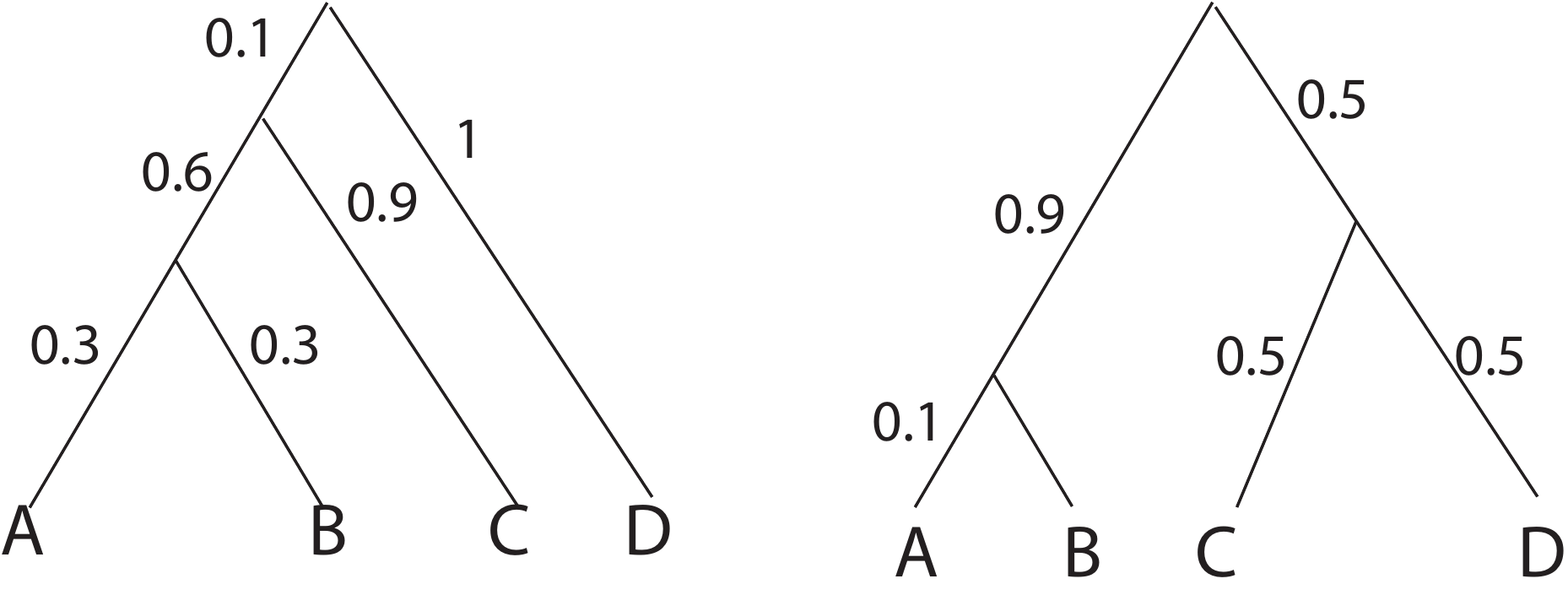}
      \caption{Examples of equidistant trees with $4$ leaves and with their height equal to $1$.}
      \label{fig:equidistant}
  \end{figure}
  \end{example}
  
  \begin{definition}[Dissimilarity Map]\label{def:dissimilarity}
  A \emph{dissimilarity map} $w$ is a function $w:[N]\times [N]\to \mathbb{R}_{\geq 0}$ such that 
  $$
    w(i,i)=0 \mbox{  and  } w(i,j)=w(j,i)\geq 0
  $$ 
    for every $i, \, j\in [N]$. If a dissimilarity map $w$ additionally satisfies the triangle inequality, $w(i,j)\leq w(i,k)+w(k,j)$ for all $i, \, j, \, k\in [N]$, then $w$ is called a \emph{metric}.
  If there exists a phylogenetic tree $T$ such that $w(i, j)$ corresponds a total branch length of the edges in the unique path from a leaf $i$ to a leaf $j$ for all leaves $i, j \in [N]$, then we call $w$ a {\em tree metric}.  If a metric $w$ is a tree metric and $w(i, j)$ corresponds the total branch length of all edges in the path from a leaf $i$ to a leaf $j$ for all leaves $i, j \in [N]$ in a phylogenetic tree $T$, then we say $w$ realises a phylogenetic tree $T$.
  \end{definition}
  In this paper, we interchangeably write $w_{ij} = w(i,j)$.  Since $w$ is symmetric, i.e., $w(i, j) = w(j, i)$ and since $w(i, j) = 0$ if $i = j$, we write 
  \[
    w = \left(w(1, 2), w(1, 3), \ldots , w(N-1, N)\right).
    \]
  
  \begin{definition}[Three Point Condition]
  If a metric $w$ satisfies the following condition: For every distinct leaves $i, j, k \in [N]$,
  \[
    \max\{w(i, j), \, w(i, k), \, w(j, k)\}
    \]
  achieves twice, then we say a metric $w$ satisfies the three point condition.
  \end{definition}
  
  \begin{definition}[Ultrametrics]\label{def:ultra}
  If a metric $w$ satisfies the three point condition then we call $w$ an {\em ultrametric}.
  \end{definition}
  
  \begin{theorem}[Proposition 12 in \cite{monod2019}]       
  \label{equid:ultra}
  A dissimilarity map $w: [N] \times [N]$ is an ultrametric if and only if $w$ is realisable of an equidistant tree with labels $[N]$.  Also there is one-to-one relation between an ultrametric $w: [N] \times [N]$ and an equidistant tree with labels $[N]$. 
  \end{theorem}
  \begin{example}\label{ex:diss}
  For equidistant trees in Figure \ref{fig:equidistant}, the dissimilarity map for the left tree in Figure \ref{fig:equidistant} is 
  \[
  (0.6, 1.8, 2, 1.8, 2, 2),
  \]
  and  the dissimilarity map for the right tree in Figure \ref{fig:equidistant} is 
  \[
  (0.2, 2, 2, 2, 2, 1).
  \]
 Note that these dissimilarity maps are tree metrics since these dissimilarity maps are computed from trees and they are also ultrametrics since they satisfy the three point condition.  
  \end{example}

  From Theorem \ref{equid:ultra} we consider the space of ultrametrics with labels $[N]$ as a space of all equidistant trees with  labels $[N]$.  Let~$\mathcal{U}_N$ be the space of ultrametrics for the equidistant trees with leaf labels $[N]$.  
  In fact we can write $\mathcal{U}_N$ as the tropicalization of the linear space generated by linear equations.

  Let $d := {N \choose 2}$, and let $L_N \subseteq \mathbb{R}^d$ be the linear subspace 
  defined by the linear equations such that
  \begin{equation}
  \label{eq:trop_eq}
  x_{ij} - x_{ik} + x_{jk}=0
  \end{equation} 
  for $1\leq i < j < k \leq N$.  
  For the linear equations (\ref{eq:trop_eq}) spanning $L_N$, their max-plus {\em tropicalization} $Trop(L_N)$ is the set of points $w$ such that $\max\left\{w_{ij}, \, w_{ik}, \, w_{jk}\right\}$ achieves at least twice for all $i, j, k \in [N]$ (see e.g., \cite{bo2017}).  This is the three point condition defined in Definition \ref{def:ultra}.

  \begin{theorem}[Theorem 3 in \cite{YZZ}]
  \label{tropicalLine}
  The image of  $~\mathcal{U}_N$ in the tropical projective torus $\R^{d}/\R\one$ coincides with $\Trop(L_N)$, where $d = {N \choose 2}$. 
  \end{theorem}
  
\begin{example}
  For $N = 3$, then there are three tree topologies for equidistant trees.  The space of ultrametrics $\mathcal{U}_3$ corresponding to the equidistant trees with $3$ leaves is a union of polyhedral cones $C_1, \, C_2, \, C_3$ in $\mathbb{R}^3$ defined by:
  \[
  \begin{array}{cc}
 C_1:& \left\{(x_{11}, x_{12}, x_{23})| x_{11} = x_{12}, x_{11} \geq x_{23}, x_{12} \geq x_{13}, x_{11} \geq 0, x_{12} \geq 0, x_{23} \geq 0, )\right\},\\
  C_2:& \left\{(x_{11}, x_{12}, x_{23})| x_{11} = x_{23}, x_{11} \geq x_{12}, x_{23} \geq x_{12}, x_{11} \geq 0, x_{12} \geq 0, x_{23} \geq 0, )\right\},\\
C_3:& \left\{(x_{11}, x_{12}, x_{23})| x_{12} = x_{23}, x_{12} \geq x_{11}, x_{23} \geq x_{11}, x_{11} \geq 0, x_{12} \geq 0, x_{23} \geq 0, )\right\}.\\
  \end{array}
  \]
  $C_1, \, C_2, \, C_3$ are the $1$-dimensional cones in $\mathbb{R}^3$.
  \end{example}

\section{Tropical SVMs}\label{sec:th}
Since the image of an equidistant tree with  labels $[N]$ under the dissimilarity map is a point in the 
tropical projective torus $\mathbb R^d \!/\mathbb R {\bf 1}$ (see e.g.,  Example \ref{ex:diss}), in this section, we consider $\mathbb R^d \!/\mathbb R {\bf 1}$ and introduce topical SVMs in $\mathbb R^d \!/\mathbb R {\bf 1}$. 


\begin{definition}[Tropical Hyperplane]
 For any $\omega:=(\omega_1, \ldots, \omega_d)\in \mathbb R^d$, the {\em tropical hyperplane defined by $\omega$}, denoted by $H_{\omega}$, is the set of points $x\in \mathbb R^d \!/\mathbb R {\bf
    1}$ such that 
  $$\max \{\omega_1+x_1, \ldots \omega_d+x_d\}$$
    is attained at least twice. We call $\omega$ the {\em normal vector} of $H_{\omega}$.
    \end{definition}
    
\begin{definition}[Sectors of Tropical Hyperplane]
    Each tropical hyperplane $H_{\omega}$ divides the tropical projective torus $\mathbb R^d \!/\mathbb R {\bf
      1}$ into $n$ connected components, which are {\em open sectors} 
  $$S_{\omega}^i~:=~\{\;x\in \mathbb R^d \!/\mathbb R {\bf
    1}\;|\; \omega_i+x_i>\omega_j+x_j,\;\forall j\neq i\;\},\;\;i=1,\ldots,d.$$
    Accordingly, we define {\em closed sectors} as 
    $$\overline{S}_{\omega}^i~:=~\{\;x\in \mathbb R^d \!/\mathbb R {\bf
    1}\;|\; \omega_i+x_i\geq \omega_j+x_j,\;\forall j\neq i\;\},\;\;i=1,\ldots,d.$$
    \end{definition}
    
    \begin{definition}[Tropical Distance to a Tropical Hyperplane]
  The {\em tropical distance} from a point $x\in \mathbb R^d \!/\mathbb R {\bf
    1}$ to a tropical hyperplane $H_{\omega}$ is defined as
  $$d_{\rm tr}(x, H_\omega)\;:=\;\min\{d_{\rm tr}(x, y)\;|\;y\in H_{\omega}\}.$$
  \end{definition}

    \begin{proposition}[Lemma 2.1 in \cite{Gartner}]\label{pp:phdis}
    Let $H_{\bf 0}$ denote the tropical hyperplane defined by the zero vector ${\bf 0}\in {\mathbb R}^d$.
  For any $x\in \mathbb R^d \!/\mathbb R {\bf
    1}$, 
  the tropical distance $d_{\rm tr}(x, H_{\bf 0})$ is the difference between the largest and the second largest coordinate of $x$. 
    \end{proposition}
  
  \begin{corollary}[Corollary  2.3 in  \cite{Gartner}]\label{cry:phdis}
  For any $\omega\in {\mathbb R}^d$, and for any $x\in \mathbb R^d \!/\mathbb R {\bf
    1}$, 
  the tropical distance $d_{\rm tr}(x, H_\omega)$ is 
  equal to $d_{\rm tr}(\omega+x, H_{\bf 0})$. 
  \end{corollary}
  
  \begin{example}\label{Ex:drpH}
    Suppose $x=(1, 2, 0)\in \mathbb R^3 \!/\mathbb R {\bf
    1}$. By Proposition \ref{pp:phdis}, $d_{\rm tr}(x, H_{\bf 0}) = 2 - 1 =1$. See 
    $x$ and $H_{\bf 0}$ in Figure \ref{fig:distancetohyperplane}.   
      \begin{figure}
      \centering
      \includegraphics[width=0.4\textwidth]{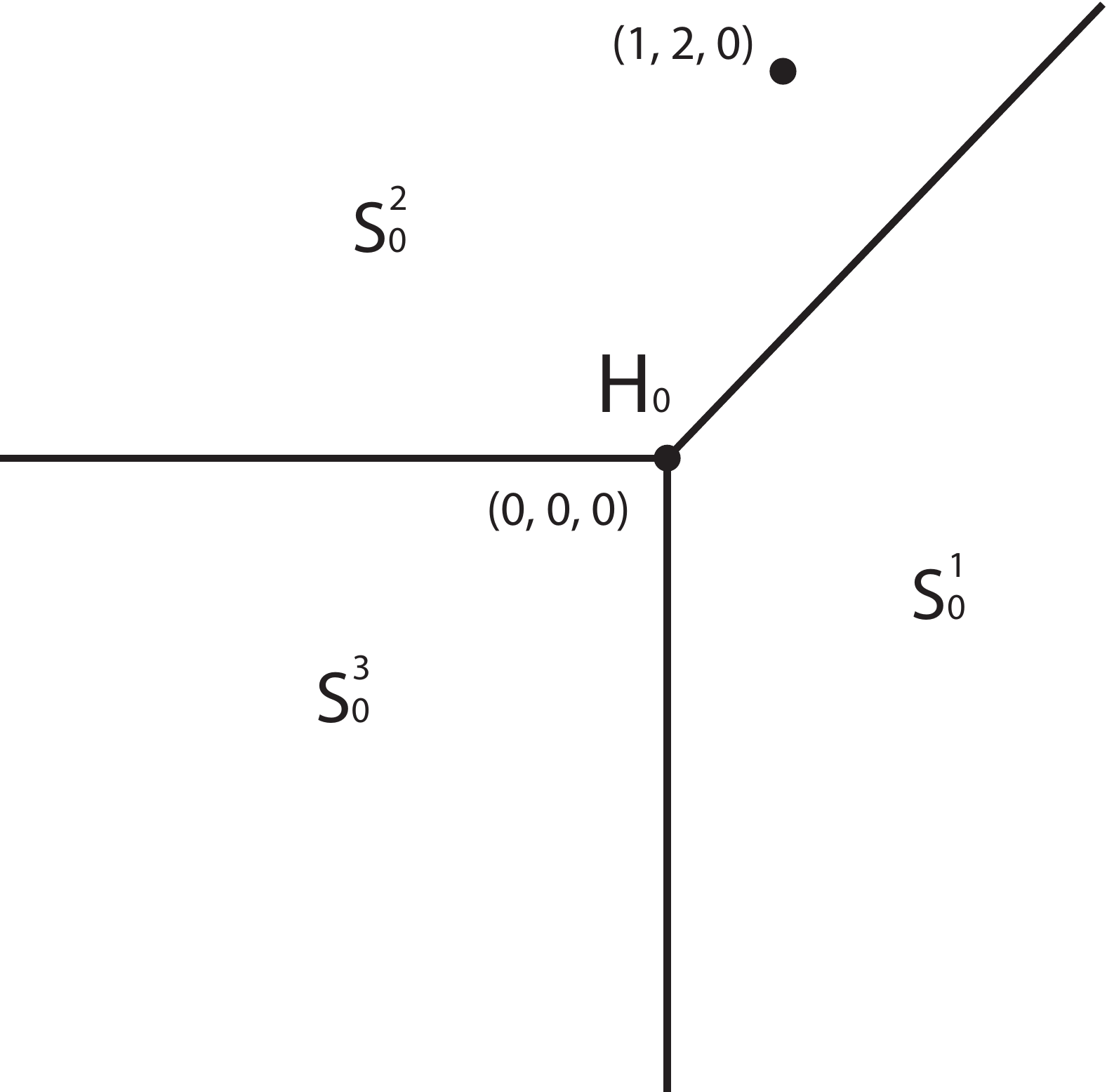}
      \caption{The tropical hyperplane $H_{\bf 0}$  in $\mathbb R^3 \!/\mathbb R {\bf 1}$}
      \label{fig:distancetohyperplane}
  \end{figure}
    \end{example}

   Assume $P$ and $Q$ ($P\cap Q=\emptyset$ and $|P|=|Q|=n$) are subsets of $\mathbb R^d \!/\mathbb R {\bf
    1}$. Suppose we have a dataset
$\{(x^{(1)}, y_1), \ldots (x^{(2n)}, y_{2n}) \}$,
where $x^{(1)}, \ldots , x^{(2n)} \in \mathbb{R}^d \!/\mathbb R {\bf 1}$, and $y_k$ is a binomial response variable such that if $y_k=0$, then $x^{(k)} \in P$ and if $y_k=1$, then 
   $x^{(k)}\in Q$.
        Our goal is to find a tropical hyperplane $H_\omega$ such that
   $P$ and $Q$ can be separated by the hyperplane and the minimum distance from the  points in $P \cup Q$ to the hyperplane $H_\omega$ can be maximized.  
    Recall that in Euclidean spaces, two categories of data  might be linearly separable (i.e., two categories can be strictly separated by a hyperplane, see \cite[Page 514]{dtamining}) or 
   nonseparable. So a classical SVM in a Euclidean space  has two versions of formulations: hard margin and soft margin. Similarly, in this section,  we discuss
   two  formulations of tropical SVMs: hard margin (Section \ref{sec:hardmargin}) and soft margin (Section \ref{sec:softmargin}). 
    
    \subsection{Hard Margin}\label{sec:hardmargin}
    
    In this section, we introduce hard margin tropical SVMs for classifying two  ``separable" sets in   $\mathbb R^d \!/\mathbb R {\bf
    1}$. 
   First, we formally define tropically separable  sets  and tropical separating hyperplanes in $\mathbb R^d \!/\mathbb R {\bf
    1}$ (see Deifnition \ref{def:sep}). 
   Similar to linearly separable data in Euclidean spaces, each point in $\mathbb R^d \!/\mathbb R {\bf
    1}$ from tropically separable sets should stay in an open sector of a tropical separating hyperplane (see (i) in Definition \ref{def:sep}), 
    and any two points from different categories should stay in different open sectors (see (ii) in Definition \ref{def:sep}). 
    \begin{definition}[Tropically Separable Sets and Tropical Separating Hyperplane]\label{def:sep}
   For any two finite sets $P$ and $Q$ in $\mathbb R^d \!/\mathbb R {\bf
    1}$, 
    if there exists $\omega \in {\mathbb R}^d$ such that  
    \begin{itemize}
    \item[(i)] for any $\xi \in P\cup Q$, there exists an index $i\in \{1, \ldots, d\}$ such that 
     \[ \text{for any}\; j\in \{1, \ldots, d\}\backslash\{ i\},\;\; \omega_{i}+\xi_{i}\;>\;\omega_{j}+\xi_{j}, \;\;\;\text{and}\]
    \item[(ii)] for any $p \in P$, and for any $q\in Q$, we have 
    $$\argmax \limits_{1\leq k\leq d} \;\{\omega_{k}+p_{k}\} \;\neq\; \argmax \limits_{1\leq k\leq d}\; \{\omega_{k}+q_{k}\},$$
    \end{itemize}
    then we say $P$ and $Q$ are {\em tropically separable}, and we say $H_\omega$ is a {\em tropical separating hyperplane} for $P$ and $Q$. 
     \end{definition}
   We introduce hard margin tropical SVMs as follows. 
   Given two finite and  tropically separable sets $P$ and $Q$ in $\mathbb R^d \!/\mathbb R {\bf
    1}$,  assume $|P|=|Q|=n>0$ (remark that all our results can be directly extended when $|P|\neq |Q|$).  For any $\xi\in P\cup Q$, we denote by $i(\xi)$ and $j(\xi)$ two indices
  in terms of $\xi$, which are integers in the set $\{1, \ldots, d\}$. 
  We denote  the two sets of indices    $\{i(\xi)| \xi \in P\cup Q\}$ and  $\{j(\xi)| \xi \in P\cup Q\}$ by 
  ${\mathcal I}$ and ${\mathcal J}$, respectively. 
  We also assume that
  
   \noindent\begin{tabularx}{\textwidth}{@{}XXX@{}}
\begin{equation}\label{equation:255}
 \forall \xi \in P\cup Q, i(\xi)\neq j(\xi), \; \;\;\;\;\text{and}\;
\end{equation}&
\begin{equation}\label{equation:2551}
  \forall p\in P, \; \forall q\in Q, \;\; i(p) \neq i(q). 
\end{equation}
\end{tabularx}
  We formulate an optimization problem\footnote{Our formulation is modified from the original formulation proposed in \cite[Section 3.1]{Gartner}, which computes an optimal tropical hyperplane through one set of data points in tropical projective spaces. Notice that in their setting, they are using min-plus algebra while we are using max-plus algebra.} for solving the normal vector $\omega$ of an optimal tropical separating hyperplane $H_{\omega}$ for $P$ and $Q$:
    \begin{equation*}\label{equation:24}
  \begin{matrix}
  \displaystyle &\max \limits_{\omega \in \mathbb{R}^{d}}\;
  \min \limits_{\xi\in P\cup Q}\;\{\xi_{i(\xi)}+\omega_{i(\xi)}-\xi_{j(\xi)}-\omega_{j(\xi)}\} \\
  \textrm{s.t.} & \forall \xi \in P\cup Q,\;\forall l \neq i(\xi), j(\xi),\;\; (\xi+\omega)_{l}\leq (\xi+\omega)_{j(\xi)}\leq(\xi+\omega)_{i(\xi)}.
  \end{matrix}
  \end{equation*}
By the constraints above we mean that   $i(\xi)$ and $j(\xi)$ respectively give the largest and the second largest coordinate of the vector $\xi+\omega$ for each $\xi \in P\cup Q$  (the assumption \eqref{equation:255} ensures these two indices are different), and 
 the object is to maximize the minimum distance $$d_{\rm tr}(\xi, H_\omega)\;=\;\xi_{i(\xi)}+\omega_{i(\xi)}-\xi_{j(\xi)}-\omega_{j(\xi)}.$$ 
  Note that this optimization  problem can be explicitly written as a linear programming problem \eqref{equation:251}--\eqref{equation:254} below, where the optimal solution $z$ means the {\em margin} of the tropical SVM (i.e., the shortest distance from the data point to the tropical separating hyperplane):
    \begin{align}
  &\max \limits_{z \in \mathbb{R}} \; z \label{equation:251} \\
  \textrm{s.t.}\;\;  \forall \xi \in P\cup Q, &\;\;z+\textcolor{black}{\xi_{j(\xi)}}+\omega_{j(\xi)}\textcolor{black}{-\xi_{i(\xi)}}-\omega_{i(\xi)}\leq 0,  \label{equation:252}\\
  \forall \xi \in P\cup Q, &\;\; \omega_{j(\xi)}-\omega_{i(\xi)}\leq \xi_{i(\xi)}-\xi_{j(\xi)}, \label{equation:253} \\
  \forall \xi \in P\cup Q, &\;\forall l \neq i(\xi), j(\xi),\;\; \omega_{l}-\omega_{j(\xi)}\leq \xi_{j(\xi)}-\xi_{l}. \label{equation:254}
  \end{align}
  
 \begin{remark}
  It is straightforward to see  
 \eqref{equation:252} and \eqref{equation:253} imply that the maximum $z$ must be non-negative. So, 
 we do not add the constraint $z\geq 0$ into the above linear programing problem. 
 \end{remark} 

  
  \begin{definition}[Feasibility and Optimal Solution (hard margin)]\label{def:feasible}
Suppose we have two sets $P$ and $Q$ in $\mathbb R^d \!/\mathbb R {\bf
    1}$, and assume sets of indices ${\mathcal I}:=\{i(\xi)|\xi\in P\cup Q\}$ and ${\mathcal J}:=\{j(\xi)|\xi \in P\cup Q\}$ satisfy the conditions \eqref{equation:255}--\eqref{equation:2551}. 
 If there exists $(z; \omega)\in {\mathbb R}^{d+1}$ such that the inequalities \eqref{equation:252}--\eqref{equation:254} hold, then we say 
$(z; \omega)\in {\mathbb R}^{d+1}$ is a {\em feasible solution} to the linear programming problem \eqref{equation:251}--\eqref{equation:254} , and
we say $P$ and $Q$ are {\em feasible} with respect to (w.r.t.) ${\mathcal I}$ and ${\mathcal J}$.  
If $(z; \omega)\in {\mathbb R}^{d+1}$ is a feasible solution such that the objective function in \eqref{equation:251} reaches its maximum value, then we say 
$(z; \omega)$ is an {\em optimal solution} to the linear programming problem. 
\end{definition}

Note that 
if $P$ and $Q$ are not tropically separable, 
 then they might not be feasilbe w.r.t. any sets of indices ${\mathcal I}$ and ${\mathcal J}$. 
On the other hand, if  $P$ and $Q$ are tropically separable, Theorem \ref{thm:sep} below ensures that 
 they are feasible w.r.t. some ${\mathcal I}$ and  ${\mathcal J}$, which also explains why
 the  hard margin tropical SVMs we propose here are indeed the analogies of hard margin SVMs in Euclidean spaces.

 \begin{theorem}\label{thm:sep}
 If two finite sets $P$ and $Q$  in $\mathbb R^d \!/\mathbb R {\bf
    1}$ are tropically separable, then there exist  $\omega\in \mathbb R^d$ and two sets of indices ${\mathcal I}:=\{i(\xi)| \xi \in P\cup Q\}$ and  ${\mathcal J}:=\{j(\xi)| \xi \in P\cup Q\}$ such that 
    the constraints \eqref{equation:255}--\eqref{equation:2551} and  \eqref{equation:252}--\eqref{equation:254} are satisfied. More than that, if $(z^*; \omega^*)$ is an optimal solution to \eqref{equation:251}--\eqref{equation:254} w.r.t. ${\mathcal I}$ and ${\mathcal J}$,
    then  the margin $z^*$ (i.e., the optimal value of \eqref{equation:251}) is positive, and $H_{\omega^*}$ is a separating tropical  hyperplane for $P$ and $Q$. 
       \end{theorem}
      \begin{proof}
 In fact, 
 by Definition \ref{def:sep},  there exist indices ${\mathcal I}$ and  ${\mathcal J}$ and a tropical separating hyperplane $H_{\omega}$ such that 
 \eqref{equation:255}--\eqref{equation:2551} and  \eqref{equation:252}--\eqref{equation:254} are satisfied. 
  The condition (i) in Definition \ref{def:sep} ensures that the distance $d_{\rm tr}(\xi, H_\omega)$ is nonzero for any $\xi \in P\cup Q$, 
 and hence, the value of $z$ corresponding to $\omega$ (i.e., the minimum distance from the points in $P\cup Q$ to $H_{\omega}$) is positive. 
  If $(z^*; \omega^*)$ is an optimal solution  w.r.t. ${\mathcal I}$ and ${\mathcal J}$, then $z^*\geq z>0$. Note $(z^*;\omega^*)$ is also a feasible solution, so  \eqref{equation:252}--\eqref{equation:254} also hold
  for $\omega^*$ w.r.t. ${\mathcal I}$ and ${\mathcal J}$. So, the condition (i) in Definition \ref{def:sep} holds for $\omega^*$. By  \eqref{equation:2551}, 
 any two points $p$ and $q$ from different sets will be located in different open sectors $S^{i(p)}_{\omega^*}$ and $S^{i(q)}_{\omega^*}$. So, the 
 the condition (ii) in Definition \ref{def:sep} is also satisfied for $\omega^*$. Therefore, $H_{\omega^*}$ is a separating tropical  hyperplane for $P$ and $Q$. 
 \end{proof}

  \begin{example}\label{Ex:lp}
Here we demonstrate how to formulate the linear programming \eqref{equation:251}--\eqref{equation:254} for a pair of tropically separable sets $P$ and $Q$.
Suppose we have four trees (with four leaves each) shown in Figure \ref{fig:four_leaves}, which form two categories $P=\{p^{(1)}, p^{(2)}\}$ and $Q=\{q^{(1)}, q^{(2)}\}$. 
The corresponding ultrametrics are
\[
\begin{array}{cc}
  p^{(1)}:(4, 10  , 20 ,  10 ,  20  , 20), \;\;  & p^{(2)}:(8  , 16  , 20  , 16 ,  20 ,  20) \\
  q^{(1)}:(2 ,  20 ,  20 ,  20  , 20 ,  10), \;\;  & q^{(2)}:(6,   20 ,  20,   20 ,  20  , 18).
\end{array}{}
\]
\begin{figure}[!ht]
    \centering
    \includegraphics[width=0.6\textwidth]{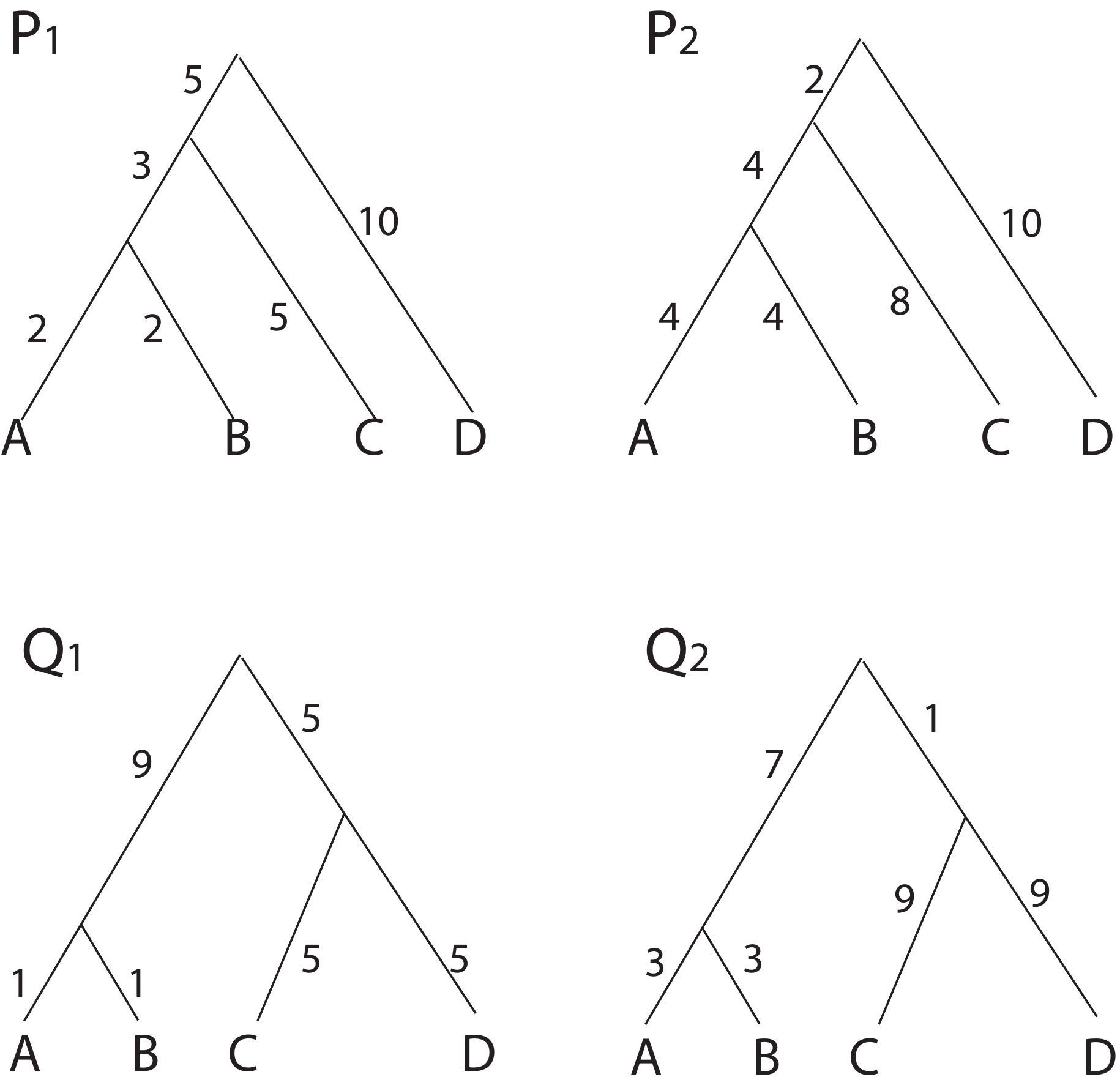}
    \caption{Four trees in Example \ref{Ex:lp}.}
    \label{fig:four_leaves}
\end{figure}
Here, for $k=1, 2$, we set $i(p^{(k)})=5$, $j(p^{(k)})=6$, $i(q^{(k)})=4$ and $j(q^{(k)})=2$.
 The linear programming \eqref{equation:251}--\eqref{equation:254}  becomes the following:
 {\tiny
\begin{align*}
\max \limits_{z \in \mathbb{R}} s.t.
  \begin{array}{cccc}
     z+20+\omega_{6}-20-\omega_{5}\leq0, \;&  z+20+\omega_{6}-20-\omega_{5}\leq0,\;&z+20+\omega_{2}-20-\omega_{4}\leq0, \;&\; z+20+\omega_{2}-20-\omega_{4}\leq0,\\
     \omega_{6}-\omega_{5}\leq 20-20,  \; &\omega_{6}-\omega_{5}\leq 20-20,\;&\omega_{2}-\omega_{4}\leq 20-20,  \; &\omega_{2}-\omega_{4}\leq 20-20,\\
     \omega_{1}-\omega_{6}\leq 20-4, \;  &\omega_{1}-\omega_{6}\leq 20-8,\;&\omega_{1}-\omega_{2}\leq 20-2, \;  &\omega_{1}-\omega_{2}\leq 20-6,\\
     \omega_{2}-\omega_{6}\leq 20-10, \;  &\omega_{2}-\omega_{6}\leq 20-16, \;&\omega_{3}-\omega_{2}\leq 20-20, \; &\omega_{3}-\omega_{2}\leq 20-20,\\
     \omega_{3}-\omega_{6}\leq 20-20, \; &\omega_{3}-\omega_{6}\leq 20-20,\;& \omega_{5}-\omega_{2}\leq 20-20, \; & \omega_{5}-\omega_{2}\leq 20-20,\\
     \omega_{4}-\omega_{6}\leq 20-10, \; & \omega_{4}-\omega_{6}\leq 20-16,\;&\omega_{6}-\omega_{2}\leq 20-10, \; &\omega_{6}-\omega_{2}\leq 20-18.
     \end{array}
\end{align*}
}By solving this linear programming via the {\tt lpSolve} package in {\tt R} \cite{lpsolve}, we would obtain the optimal solution $(z^*; \omega^*)$, where $z^*=2$ and 
$\omega^* = (0,2, 0, 4, 2, 0)$. It is straightforward to check 
that points in $P$ and $Q$ are located in the open sectors $S_{\omega^*}^5$ and $S_{\omega^*}^4$, respectively.
\end{example}{}

\begin{example}\label{ex:non-feasible}
The sets $P$ and $Q$ in Example \ref{Ex:lp} might not be feasible w.r.t. some sets of indices ${\mathcal I}$ and ${\mathcal J}$. 
For instance, if we replace the indices in Example \ref{Ex:lp} with
\[
i(p^{(k)})=5, \;\;  j(p^{(k)})=6, \;\;  i(q^{(k)})=4 \;\;  and  \;\; j(q^{(k)})=2 \;(k=1, 2),
\]
then the linear programming will have no feasible solution.  
\end{example}

\begin{remark}
Example \ref{Ex:lp} and Example \ref{ex:non-feasible} show that even tropically separable sets $P$ and $Q$ might not be feasible for some sets of indices ${\mathcal I}$ and ${\mathcal J}$. 
Theorem \ref{thm:sep} shows feasible sets of indices must exist for two tropically separable sets (for instance, the sets $P$ and $Q$ in \ref{Ex:lp} are feasible w.r.t. the indices shown in that example). 
But Theorem \ref{thm:sep} does not say how to find ${\mathcal I}$ and ${\mathcal J}$ such that $P$ and $Q$ are feasible w.r.t. ${\mathcal I}$ and ${\mathcal J}$. 
In general, in order to find the feasible indices, 
we need go over all possible choices for those indices. That means we need to solve $d^{2n}(d-1)^{2n}$ linear programming problems in $d+1$ variables with 
$dn$ constraints (recall that $d=\binom{N}{2}$, where $N$ is the number of leaves, and 
recall that  $n$ is the cardinality of $P$ and $Q$).  
\end{remark}

In order to avoid computations on the non-feasible cases,   we would like to ask
 {\em under what conditions (tropically separable)  sets $P$ and $Q$ will be feasible w.r.t. two sets of indices ${\mathcal I}$ and ${\mathcal J}$?} 
 In the rest of this section, we answer this question when for all $p\in P$,
$i(p)$ and $j(p)$ are constants, say $i_P$ and $j_P$ and for all $q\in Q$, $i(q)$ and $j(q)$ are constants, say $i_Q$ and $j_Q$ (as what has been shown in Example \ref{Ex:lp}). Geometrically, this means that
the data points in $P$ and those in $Q$ will respectively stay in the same sector determined by a tropical separating hyperplane. Notice that  
if the four indices $i_P, j_P, i_Q$ and $j_Q$ satisfy the assumptions \eqref{equation:255}--\eqref{equation:2551}  (i.e. $i_P\neq j_P$,  $i_Q\neq j_Q$ and $i_P\neq i_Q$), then there are four cases:
\begin{itemize}
\item[]({\bf Case 1}). the four indices $i_P, j_P, i_Q$ and $j_Q$ are pairwise distinct; 
\item[]({\bf Case 2}).  $i_{P}=j_{Q}$ and $i_{Q}\neq j_P$, or, $i_{Q}=j_{P}$ and $i_{P}\neq j_Q$;
\item[]({\bf Case 3}).  $i_{P}=j_{Q}$ and $i_{Q}= j_P$;
\item[]({\bf Case 4}).  $j_Q=j_P$.
\end{itemize}

For each case above, we provide a sufficient and necessary condition for the feasibility of the sets $P$ and $Q$, and an explicit 
formula for the optimal value $z$. See Theorems \ref{thm:16}--\ref{theorem76}. We present the proofs in Appendix. 
If the given sets are not  tropically separable (it indeed happens a lot in practice), 
then we should apply the soft margin SVMs discussed in Section \ref{sec:softmargin}. 

\begin{remark}
Remark that in Theorems \ref{thm:16}--\ref{theorem76}, we do not require $P$ and $Q$ to be tropically separable. 
\end{remark}

\begin{theorem}\label{thm:16}
Suppose  $P$ and $Q$ are two finite sets in $\mathbb R^d \!/\mathbb R {\bf
    1}$. 
For all $p\in P$,
assume $i(p)$ and $j(p)$ are constants, say $i_P$ and $j_P$. For all $q\in Q$, assume $i(q)$ and $j(q)$ are constants, say $i_Q$ and $j_Q$.
If the four numbers $i_P, j_P, i_Q$ and $j_Q$ are pairwise distinct, then the linear programming \eqref{equation:251}--\eqref{equation:254} has a feasible solution if and only if
\begin{align}
  \max \{-F, \;-A-E\} \;\leq \;\min\{D+B, \; C\}. 
   \label{eq:thm163}
\end{align}
If a feasible solution exists, then the optimal value $z$ is given by
\begin{align}\label{eq:thm164}
\min\;\{\;  A+C+E,\; D+B+F, \; \frac{1}{2}\left(A+B+D+E\right) \;\},
\end{align}
where 
\begin{align}\label{eq:thm165}
\begin{array}{ccc}
A=\min\limits_{p \in P}\{p_{i_{P}}-p_{j_{P}}\},   &  C=\min\limits_{p \in P}\{p_{j_{P}}-p_{j_{Q}}\},  & E=\min\limits_{q\in Q}\{q_{j_Q}-q_{i_P}\}, \\
B=\min\limits_{p\in P}\{p_{j_{P}}-{p_{i_{Q}}}\}, &  D=\min\limits_{q\in Q}\{q_{i_{Q}}-{q_{j_{Q}}}\},  &  F=\min\limits_{q\in Q}\{q_{j_Q}-q_{j_P}\}.
  \end{array}
 \end{align}
\end{theorem}

\begin{theorem}\label{thm:17}
Suppose  $P$ and $Q$ are two finite sets in $\mathbb R^d \!/\mathbb R {\bf
    1}$. 
For all $p\in P$,
assume $i(p)$ and $j(p)$ are constants, say $i_P$ and $j_P$. For all $q\in Q$, assume $i(q)$ and $j(p)$ are constants, say $i_Q$ and $j_Q$.

\begin{itemize}
    \item[(i)] If $i_{P}=j_{Q}$ and $i_{Q}\neq j_P$, then
   linear programming \eqref{equation:251}--\eqref{equation:254} has a feasible solution if and only if
   \begin{align}\label{eq:thm171}
   A+B+C\geq 0.
\end{align}

If a feasible solution exists, then the optimal value $z$ is given by
\begin{align*}
\min\;\{\;  A+B+C,\; \frac{1}{2}\left(A'+B+C\right) \},
\end{align*}
where 
\begin{align*}
\begin{array}{cc}
A'=\min\limits_{p \in P}\{p_{i_{P}}-p_{j_{P}}\},   &  A=\min\limits_{\xi \in P\cup Q}\{\xi_{i_{P}}-\xi_{j_{P}}\}, \\
B=\min\limits_{p\in P}\{p_{j_{P}}-{p_{i_{Q}}}\}, &  C=\min\limits_{q\in Q}\{q_{i_{Q}}-{q_{i_{P}}}\}.
  \end{array}
 \end{align*}
\item[(ii)] If $i_{Q}=j_{P}$ and $i_{P}\neq j_{Q}$, then
linear programming \eqref{equation:251}--\eqref{equation:254} has a feasible solution if and only if
\begin{align}\label{eq:thm173}
A+B+C\geq 0 .
\end{align}
If a feasible solution exists, then the optimal value $z$ is given by
\begin{align*}
\min\;\{\; A+B+C,\; \frac{1}{2}\left(A'+B+C\right)\; \},
\end{align*}
where 
\begin{align*}
\begin{array}{cc}
A'=\min\limits_{q \in Q}\{q_{i_{Q}}-q_{j_{Q}}\},   &  A=\min\limits_{\xi \in P\cup Q}\{\xi_{i_{Q}}-\xi_{j_{Q}}\}, \\
                                     B=\min\limits_{q\in Q}\{q_{j_{Q}}-{q_{i_{P}}}\}, &  C=\min\limits_{p\in P}\{p_{i_{P}}-{p_{i_{Q}}}\}.
                                     \end{array}
                                     \end{align*}
                                     \end{itemize}
                                     \end{theorem}


 \begin{theorem}\label{thm:15}
 Suppose  $P$ and $Q$ are two finite sets in $\mathbb R^d \!/\mathbb R {\bf
    1}$. 
  If for all $p\in P$,
  we have $i(p) = j(q)=k_1$, and for all $q\in Q$,  $i(q)=j(p)=k_2$, then the linear programming \eqref{equation:251}--\eqref{equation:254} has a feasible solution if and only if
  \begin{align}\label{eq:thm151}
  \max_{p \in P}\{p_{k_2}-p_{k_1}\} \leq \min\limits_{q \in Q}\{q_{k_2}-q_{k_1}\} 
  \end{align}
  If a feasible solution exists, then the optimal value $z$ is given by
  \begin{align}\label{eq:thm152}
  \frac{1}{2}(\min\limits_{p \in P}\{p_{k_1}-p_{k_2}\} + \min\limits_{q \in Q}\{q_{k_2}-q_{k_1}\})
  \end{align}
  \end{theorem}{}

                            \begin{theorem}\label{theorem76}
                            Suppose  $P$ and $Q$ are two finite sets in $\mathbb R^d \!/\mathbb R {\bf
    1}$. 
                            If for all $p\in P$ and for all $q\in Q$, the indices $i(p)$ and $i(q)$ are respectively constants, say $i_P$ and $i_Q$,
                            and for all $\xi \in P \cup Q$, $j(\xi)$ is a constant, say $j$, then the linear programming \eqref{equation:251}--\eqref{equation:254} has a feasible solution if and only if
                            \begin{align}\label{eq:thm141}
                            \max_{q\in Q}\{q_{i_{P}}-q_{j}\}~\leq~ \min\limits_{p\in P}\{p_{i_{P}}-p_{j}\}
                            \end{align}
                              and
                            \begin{align}\label{eq:thm142}
                            \max_{p\in P}\{p_{i_{Q}}-p_{j}\}~ \leq~ \min\limits_{q\in Q}\{q_{i_{Q}}-q_{j}\}.
                            \end{align}
                            If a feasible solution exists, then the optimal value $z$ is given by
                            \begin{align}\label{eq:thm143}
                            \min\{\;\min\limits_{p\in P}\{p_{i_{P}}-p_{j}\} + \min\limits_{q\in Q}\{q_{j}-q_{i_{P}}\}, \;\min\limits_{q\in Q}\{q_{i_{Q}}-q_{j}\} + \min\limits_{p\in P}\{p_{j}-p_{i_{Q}}\}\;\}.
                            \end{align}
                            \end{theorem}{}

   \subsection{Soft Margin}\label{sec:softmargin} 
   
   We introduce a tropical soft margin  SVM in \eqref{equation:421}--\eqref{equation:425}  by adding 
   non-negative slacker variables $\alpha_{\xi}$, $\beta_{\xi}$ and $\gamma_{\xi, l}$ into each constraint in the hard margin tropical SVM \eqref{equation:251}--\eqref{equation:254}. 
      
   \begin{align}
&  \max \limits_{\left(z; \alpha; \beta; \gamma \right) \in \mathbb{R}^{2dn+1}} \; z - {\mathcal C}\sum\limits_{\xi\in P\cup Q}\left(\alpha_{\xi}+\beta_{\xi}+\sum\limits_{l\neq i(\xi), j(\xi)}\gamma_{\xi, l}\right)  \label{equation:421} \\
  \textrm{s.t.}\;\; & \forall \xi \in P\cup Q, \;\;z+\textcolor{black}{\xi_{j(\xi)}}+\omega_{j(\xi)}\textcolor{black}{-\xi_{i(\xi)}}-\omega_{i(\xi)}\leq \alpha_{\xi},  \label{equation:422}\\
 & \forall \xi \in P\cup Q, \;\; \omega_{j(\xi)}-\omega_{i(\xi)}\leq \xi_{i(\xi)}-\xi_{j(\xi)} +\beta_{\xi}, \label{equation:423} \\
  &\forall \xi \in P\cup Q, \;\forall l \neq i(\xi), j(\xi),\;\; \omega_{l}-\omega_{j(\xi)}\leq \xi_{j(\xi)}-\xi_{l} + \gamma_{\xi, l}, \label{equation:424} \\
  &\forall \xi \in P\cup Q, \;\forall l \neq i(\xi), j(\xi),\;\;  \alpha_{\xi}\geq 0, \;\; \beta_{\xi}\geq 0, \;\;\gamma_{\xi, l} \geq 0, \;\; \text{and}\;\; z\geq 0, \label{equation:425}
  %
  \end{align}
  where $n=|P|=|Q|$,   the term $\sum\limits_{\xi\in P\cup Q}\left(\alpha_{\xi}+\beta_{\xi}+\sum\limits_{l\neq i(\xi), j(\xi)}\gamma_{\xi, l}\right)$ gives the {\em Hinge loss}
(i.e., a linear estimate of the deviation from the tropically separable case), and 
  ${\mathcal C}$  is a positive scalar that controls the trade-off between maximizing the margin $z$ or minimizing the loss (see e.g., \cite[Page 524, Formula (21.24)]{dtamining}).

          \begin{definition}[Feasibility and Optimal Solution (soft margin)]\label{def:feasiblesoft}
          Suppose we have two sets $P$ and $Q$ in $\mathbb R^d \!/\mathbb R {\bf
    1}$ $(n=|P|=|Q|)$, and assume two sets of indices ${\mathcal I}:=\{i(\xi)|\xi\in P\cup Q\}$ and ${\mathcal J}:=\{j(\xi)|\xi \in P\cup Q\}$ satisfy the conditions \eqref{equation:255}--\eqref{equation:2551}. 
If there exists $(z; \alpha; \beta; \gamma; \omega)\in {\mathbb R}^{2dn+d+1}$ such that the inequalities \eqref{equation:422}--\eqref{equation:425} hold, then we say 
$(z; \alpha; \beta; \gamma; \omega)$ is a {\em feasible solution to the linear programming \eqref{equation:422}--\eqref{equation:425}}.
If $(z; \alpha; \beta; \gamma; \omega)\in {\mathbb R}^{2dn+d+1}$ is a feasible solution such that the objective function in \eqref{equation:422} reaches its maximum value, then we say 
$(z; \alpha; \beta; \gamma; \omega)$ is an {\em optimal solution} to the linear programming problem. 
\end{definition}

\begin{proposition}\label{pp:subfeasible}
Given two  sets $P$ and $Q$ in $\mathbb R^d \!/\mathbb R {\bf
    1}$ $(n=|P|=|Q|)$, assume two sets of indices ${\mathcal I}$ and ${\mathcal J}$ satisfy the conditions \eqref{equation:255}--\eqref{equation:2551}. 
Then, the linear programming \eqref{equation:421}--\eqref{equation:425} has infinitely many  feasible solutions.  
\end{proposition}
\begin{proof}
In fact, for any $\omega^* \in \mathbb{R}^{n}$, 
let $z^*=0$, and let
$$\alpha_{\xi}^*= \beta_{\xi}^* = \max \{0, \left(\omega^*_{j(\xi)}-\omega^*_{i(\xi)}\right) -  \left(\xi_{i(\xi)}-\xi_{j(\xi)}\right)\}, \;\text{for any} \; \xi \in P\cup Q,$$
 $$\gamma_{\xi, l}^* = \max \{0, \left(\omega^*_{l}-\omega^*_{j(\xi)}\right) -  \left(\xi_{j(\xi)}-\xi_{l}\right)\}, \;\text{for any} \; \xi \in P\cup Q, \; \text{for any}\; l\neq i(
 \xi), j(\xi).$$ It is straightforward to check that $(z^*; \alpha^*; \beta^*; \gamma^*; \omega^*)\in {\mathbb R}^{2dn+d+1}$ satisfies the inequalities \eqref{equation:422}--\eqref{equation:425}. 
 \end{proof}
 
  \begin{theorem}\label{thm:boundedsoft}
  Given two finite sets $P$ and $Q$ in $\mathbb R^d \!/\mathbb R {\bf
    1}$ $(n=|P|=|Q|)$, assume two sets of indices ${\mathcal I}$ and ${\mathcal J}$ satisfy the conditions \eqref{equation:255}--\eqref{equation:2551}. 
 If  both $P$ and $Q$ are non-empty, and if ${\mathcal C}\geq 1$, then the objective function in the linear programming \eqref{equation:421}--\eqref{equation:425} is upper bounded for any feasible solution
 $(z; \alpha; \beta; \gamma; \omega)\in {\mathbb R}^{2dn+d+1}$, which means 
 the maximum of the objective function is a finite real number.  
 \end{theorem}
 
 \begin{proof}
 Since both $P$ and $Q$ are non-empty, pick $p$ and $q$ from $P$ and $Q$ respectively. 
 By the assumptions \eqref{equation:255}--\eqref{equation:2551}, we know $i(p)\neq j(p)$, $i(q)\neq j(q)$, and $i(p) \neq i(q)$. Below, we prove the conclusion for the two cases: $i(p)\neq j(q)$ and $i(p)=j(q)$.  
 
 {\bf (Case A).} If $i(p)\neq j(q)$, then by \eqref{equation:424}, 
 we have 
\begin{align}\label{proof:bs1}
 \omega_{i(p)}-\omega_{j(q)}\leq q_{j(q)}-q_{i(p)} +\gamma_{q, i(p)}.
 \end{align}
 
 \begin{itemize}
 \item[]{\bf (Case A.1).} If $j(q)=j(p)$, then \eqref{proof:bs1} becomes
 \begin{align}\label{proof:bs2}
 \omega_{i(p)}-\omega_{j(p)}\leq q_{j(p)}-q_{i(p)} +\gamma_{q, i(p)}.
 \end{align}
Then,  by \eqref{equation:422} and \eqref{proof:bs2}, 
\begin{align*}
z \leq p_{i(p)} - p_{j(p)} + q_{j(p)}-q_{i(p)} +\gamma_{q, i(p)}  + \alpha_{p}.
 \end{align*}
 So, if ${\mathcal C}\geq 1$, then we have an upper bound for the objective function 
\[ z -{\mathcal C} \sum\limits_{\xi\in P\cup Q}\left(\alpha_{\xi}+\beta_{\xi}+\sum\limits_{l\neq i(\xi), j(\xi)}\gamma_{\xi, l}\right)\leq z - \alpha_{p}-\gamma_{q, i(p)}  \leq p_{i(p)} - p_{j(p)} + q_{j(p)}-q_{i(p)} .\]

\item[]{\bf (Case A.2).} If $j(q)\neq j(p)$, then by \eqref{equation:424},
 we have 
\begin{align}\label{proof:bs3}
 \omega_{j(q)}-\omega_{j(p)}\leq p_{j(p)}-p_{j(q)} +\gamma_{p, j(q)}.
 \end{align}
 By summing up \eqref{proof:bs1} and \eqref{proof:bs3}, we have 
 \begin{align}\label{proof:bs4}
  \omega_{i(p)}-\omega_{j(p)}\leq q_{j(q)}-q_{i(p)} +\gamma_{q, i(p)}+p_{j(p)}-p_{j(q)} +\gamma_{p, j(q)}.
 \end{align}
So if ${\mathcal C}\geq 1$, then by \eqref{equation:422} and \eqref{proof:bs4}, we have an upper bound for the objective function 
\[ z -{\mathcal C} \sum\limits_{\xi\in P\cup Q}\left(\alpha_{\xi}+\beta_{\xi}+\sum\limits_{l\neq i(\xi), j(\xi)}\gamma_{\xi, l}\right)\leq 
z - \alpha_{p}-\gamma_{q, i(p)} - \gamma_{p, j(q)}\leq 
q_{j(q)}-q_{i(p)} +p_{i(p)}-p_{j(q)}.\]

 \end{itemize}
 
{\bf (Case B).}
 If $i(p) = j(q)$, then by \eqref{equation:422}, 
 \begin{align}\label{proof:bs5}
 \omega_{i(p)}-\omega_{i(q)}\leq q_{i(q)}-q_{i(p)} +\alpha_{q}.
 \end{align}
 \begin{itemize}
 \item[]{\bf (Case B.1).} If $i(q)=j(p)$, then \eqref{proof:bs5} becomes
  \begin{align}\label{proof:bs6}
 \omega_{i(p)}-\omega_{j(p)}\leq q_{j(p)}-q_{i(p)} +\alpha_{q}.
 \end{align}
 So if ${\mathcal C}\geq 1$, then by \eqref{equation:422} and \eqref{proof:bs6}, we have an upper bound for the objective function 
\[ z -{\mathcal C} \sum\limits_{\xi\in P\cup Q}\left(\alpha_{\xi}+\beta_{\xi}+\sum\limits_{l\neq i(\xi), j(\xi)}\gamma_{\xi, l}\right)\leq 
z-\alpha_{q} - \alpha_{p}\leq 
q_{j(p)}-q_{i(p)} +p_{i(p)}-p_{j(p)}.\]
 
 \item[]{\bf (Case B.2).} If $i(q)\neq j(p)$,then by \eqref{equation:424}, we have 
 \begin{align}\label{proof:bs7}
 \omega_{i(q)}-\omega_{j(p)}\leq p_{j(p)}-p_{i(q)} +\gamma_{p, i(q)}.
 \end{align}
 By summing up \eqref{proof:bs5} and \eqref{proof:bs7}, 
  \begin{align}\label{proof:bs8}
 \omega_{i(p)}-\omega_{j(p)}\leq q_{i(q)}-q_{i(p)} +\alpha_{q} + p_{j(p)}-p_{i(q)} +\gamma_{p, i(q)}.
 \end{align}
 \end{itemize}
  So if ${\mathcal C}\geq 1$, then by \eqref{equation:422} and \eqref{proof:bs8}, we have an upper bound for the objective function 
\[ z -{\mathcal C} \sum\limits_{\xi\in P\cup Q}\left(\alpha_{\xi}+\beta_{\xi}+\sum\limits_{l\neq i(\xi), j(\xi)}\gamma_{\xi, l}\right)\leq 
z-\alpha_{q} - \alpha_{p} - \gamma_{p, i(q)}\leq 
q_{i(q)}-q_{i(p)} + p_{i(p)}-p_{i(q)}.\]
 \end{proof}
 
 \begin{remark}\label{rmk:boundedsoft1}
 Although we have assumed from the every beginning that $|P|=|Q|$, it is easily seen that 
 Theorem \ref{thm:boundedsoft} still holds if $|P|\neq |Q|$.    
 We emphasize the hypothesis ``both $P$ and $Q$ are non-empty" in Theorem \ref{thm:boundedsoft} because if one of $P$ and $Q$ is empty, then the 
 objective function in  the linear programming \eqref{equation:421}--\eqref{equation:425} might not be upper bounded for ${\mathcal C}\geq 1$. For instance, 
 if $P=\{p\}$ where $p = (5,5,4,3,2,1)$ with $i(p)=1$ and $j(p)=2$,  and if $Q$ is empty, then it is directly to check that the  objective function  is not upper bounded. However, the sets $P$ and $Q$ associated with a real dataset  should be both non-empty.
 \end{remark}
 
  \begin{remark}\label{rmk:boundedsoft2}
  We remark that if the assumption ${\mathcal C}\geq 1$ in Theorem \ref{thm:boundedsoft} is not satisfied, then 
  the 
 objective function in  the linear programming \eqref{equation:421}--\eqref{equation:425} might not be upper bounded.
 And, for ${\mathcal C}< 1$, it is also possible for the objective function to be upper bounded. 
 For instance, for one pair of simulated  sets $P$ and $Q$ we have tested (see ``unbounded.RData" in Table \ref{table:files}), the objective function is not upper bounded when ${\mathcal C}<0.00625$, and it is
 upper bounded when ${\mathcal C}\geq 0.00625$. 
 \end{remark}
 
 \begin{remark}
Similarly to the hard margin tropical SVMs, for fixed sets of indices ${\mathcal I}$ and ${\mathcal J}$, we can get an optimal tropical hyperplane $H_{\omega}$ for these indices by solving the linear programming problem \eqref{equation:421}--\eqref{equation:425}. In general, 
if we go over all possible choices for these indices, then we will get $d^{2n}(d-1)^{2n}$ linear programming problems in $2dn+d+1$ variables with 
$3dn+1$ constraints. 
\end{remark}

In order to save computational time, we again consider the simpler case discussed in Section \ref{sec:hardmargin}. We assume 
for all $p\in P$,
$i(p)$ and $j(p)$ are constants, say $i_P$ and $j_P$ and for all $q\in Q$, $i(q)$ and $j(q)$ are constants, say $i_Q$ and $j_Q$. In this case, it is possible for us to simplify 
 \eqref{equation:421}--\eqref{equation:425} by removing the slacker variables $\gamma_{\xi, l}$ for $l\neq i_P, j_P, i_Q, j_Q$, see
Proposition \ref{pp:gamma0}.
 
  \begin{proposition}\label{pp:gamma0}
  For all $p\in P$,
assume $i(p)$ and $j(p)$ are constants, say $i_P$ and $j_P$. For all $q\in Q$, assume $i(q)$ and $j(q)$ are constants, say $i_Q$ and $j_Q$. 
If $(z^*; \alpha^*; \beta^*; \gamma^*; \omega^*)$ 
 is an optimal solution   to  the linear programming \eqref{equation:421}--\eqref{equation:425}, then for  any $l\neq i_P, j_P, i_Q, j_Q$, we have $\gamma_{\xi, l}^*=0$ for all $\xi\in P\cup Q$.

\end{proposition}

\begin{proof}
Assume that there exists $\tilde{l}\neq i_P, j_P, i_Q, j_Q$ such that $\gamma_{\xi, \tilde{l}}^*>0$ for some $\xi \in P\cup Q$. By \eqref{equation:424}, 
\[\omega^*_{\tilde{l}} - \omega^*_{j(\xi)}\; \leq\; \xi_{j(\xi)} -\xi_{\tilde{l}} + \gamma_{\xi, \tilde{l}}^*. \]
Let $\hat \omega^*_{\tilde{l}} = \omega^*_{\tilde{l}} -  \gamma_{\xi, \tilde{l}}^*$, and let $\hat  \gamma_{\xi, \tilde{l}}^* = 0$. We replace 
the coordinates $\omega^*_{\tilde{l}}$ and $\gamma_{\xi, \tilde{l}}^*$ of the vector $(z^*; \alpha^*; \beta^*; \gamma^*; \omega^*)$ with $\hat \omega^*_{\tilde{l}}$ and $\hat  \gamma_{\xi, \tilde{l}}^* $, and we call the resulting vector 
$(z^*; \alpha^*; \beta^*; \hat \gamma^*; \hat \omega^*)$. Then $(z^*; \alpha^*; \beta^*; \hat \gamma^*; \hat \omega^*)$ is still a feasible solution, for which
the objective function has the value
\[z^* -{\mathcal C} \sum\limits_{\xi\in P\cup Q}\left(\alpha^*_{\xi}+\beta^*_{\xi}+\sum\limits_{l\neq i(\xi), j(\xi)}\gamma^*_{\xi, l}\right) + \gamma_{\xi, \tilde{l}}^*\;>\; z^* -{\mathcal C} \sum\limits_{\xi\in P\cup Q}\left(\alpha^*_{\xi}+\beta^*_{\xi}+\sum\limits_{l\neq i(\xi), j(\xi)}\gamma^*_{\xi, l}\right) .\]
That means $(z^*; \alpha^*; \beta^*; \hat \gamma^*; \hat \omega^*)$ gives a larger function value to the objective function, which is a contradiction to the fact that  the objective function reaches its maximum at $(z^*; \alpha^*; \beta^*; \gamma^*; \omega^*)$. 
\end{proof}

\begin{corollary}\label{cry:gamma0}
 For all $p\in P$,
assume $i(p)$ and $j(p)$ are constants, say $i_P$ and $j_P$. For all $q\in Q$, assume $i(q)$ and $j(q)$ are constants, say $i_Q$ and $j_Q$. 
Then any point in $P$ is located in the closed sector $\overline{S}_{\omega^*}^{i_P}$ or in $\overline{S}_{\omega^*}^{j_P}$, and 
any point in $Q$ is located in the closed sector $\overline{S}_{\omega^*}^{i_Q}$ or in $\overline{S}_{\omega^*}^{j_Q}$ of an optimal tropical hyperplane $H_{\omega^*}$.
\end{corollary}

\begin{proof}
Suppose $(z^*; \alpha^*; \beta^*; \gamma^*; \omega^*)$  is a feasible solution   to  the linear programming \eqref{equation:421}--\eqref{equation:425} such that the objective function reaches its maximum. 
By Proposition \ref{pp:gamma0}, 
for any $p
\in P$, we have 
\[\omega^*_{l} - \omega^*_{j_P}\; \leq\; p_{j_P} -p_{l} \; \Leftrightarrow \; (p+\omega^*)_{l}\; \leq\; (p+\omega^*)_{j_P} ,\]
for  any $l\neq i_P, j_P$. Then the maximum coordinate of vector $p+\omega^*$ can  be indexed by  
$i_P$ or $j_P$. So, 
any point in $P$ is located in the closed sector $\overline{S}_{\omega^*}^{i_P}$ or in $\overline{S}_{\omega^*}^{j_P}$. Similarly, we can show that  
any point in $Q$ is located in the closed sector $\overline{S}_{\omega^*}^{i_Q}$ or in $\overline{S}_{\omega^*}^{j_Q}$.
\end{proof}

Below,  for the four cases ({\bf Case 1})--({\bf Case 4}) list in 
Section \ref{sec:hardmargin}, 
 we respectively simplify 
\eqref{equation:421}--\eqref{equation:425} as four linear programing problems \eqref{equation:lp7}-\eqref{equation:lp5} (for ({\bf Case 2}), we only show the simplified linear programming problem for $i_P=j_Q$ and $i_Q\neq j_P$).  Notice that 
we remove the slacker variables $\gamma_{\xi, l}$ for $l\neq i_P, j_P, i_Q, j_Q$ by Proposition \ref{pp:gamma0}. Also, by Theorem \ref{thm:boundedsoft}, we set 
${\mathcal C}=1$ for making sure these 
linear programming problems have bounded optimal values. 

 \begin{minipage}{.5\textwidth}
  {\scriptsize
    \begin{align}
 & \max \limits_{\left(z; \alpha; \beta; \gamma \right) \in \mathbb{R}_{\geq 0}^{8n+1}} \; z - \sum\limits_{\xi\in P\cup Q}\left(\alpha_{\xi}+\beta_{\xi}\right) \notag\\
  &\;\;\;\;\;\;\;\;-\sum\limits_{p\in P, \; l=i_Q, j_Q}\gamma_{p, l}-\sum\limits_{q\in Q,\; l=i_P, j_P}\gamma_{q, l}    \notag \\
  \textrm{s.t.}&\;\;  \forall p \in P, \;\;z+\textcolor{black}{p_{j_P}}+\omega_{j_P}\textcolor{black}{-p_{i_P}}-\omega_{i_P}\leq \alpha_{p}, \notag\\
                                               &\;\; \;\;\;\;\;\;\;\; \;\;\;\;\;\;\;\; \omega_{j_P}-\omega_{i_P}\leq p_{i_P}-p_{j_P} +\beta_{p}, \notag\\
                                                 & \;\; \;\;\;\;\;\;\;\; \;\;\;\;\;\;\;\; \omega_{i_Q}-\omega_{j_P}\leq p_{j_P}-p_{i_Q}+\gamma_{p, i_Q}, \notag\\
   & \;\; \;\;\;\;\;\;\;\; \;\;\;\;\;\;\;\;\omega_{j_Q}-\omega_{j_P}\leq p_{j_P}-p_{j_Q}+\gamma_{p, j_Q}, \notag\\
    & \;\; \;\;\;\;\;\;\;\; \;\;\;\;\;\;\;\;\;\;\;  \omega_{l}-\omega_{j_P}\leq p_{j_P}-p_{l} \notag \\  
     & \;\; \;\;\;\;\;\;\;\; \;\;\;\;\;\;\;\;\;  \;\;\;\;\;\;\;\;\; \;\;\;\;\;\;\;\;\; \;\;\;\;\;   (\forall l \neq i_P, j_P,i_Q, j_Q),\notag\\
&  \forall q \in  Q, \;\;z+\textcolor{black}{q_{j_Q}}+\omega_{j_Q}\textcolor{black}{-q_{i_Q}}-\omega_{i_Q}\leq \alpha_{q}, \notag\\
                        &\;\; \;\;\;\;\;\;\;\; \;\;\;\;\;\;\;\; \omega_{j_Q}-\omega_{i_Q}\leq q_{i_Q}-q_{j_Q} +\beta_{q}, \notag\\
 & \;\; \;\;\;\;\;\;\;\; \;\;\;\;\;\;\;\;\omega_{i_P}-\omega_{j_Q}\leq q_{j_Q}-q_{i_P}+\gamma_{q, i_P}, \notag\\
&  \;\; \;\;\;\;\;\;\;\; \;\;\;\;\;\;\;\;\omega_{j_P}-\omega_{j_Q}\leq q_{j_Q}-q_{j_P}+\gamma_{q, j_P}, \notag\\
& \;\; \;\;\;\;\;\;\;\; \;\;\;\;\;\;\;\;\;\;\; \omega_{l}-\omega_{j_Q}\leq q_{j_Q}-q_{l}  \notag \\
& \;\; \;\;\;\;\;\;\;\; \;\;\;\;\;\;\;\;\;  \;\;\;\;\;\;\;\;\; \;\;\;\;\;\;\;\;\; \;\;\;\;\;     (\forall l \neq i_P, j_P,i_Q, j_Q). \tag{{\bf LP 1}} \label{equation:lp7}
  %
  \end{align}
  }
  \end{minipage}
   \begin{minipage}{.5\textwidth}
   {\scriptsize
     \begin{align}
& \max \limits_{\left(z; \alpha; \beta; \gamma \right) \in \mathbb{R}_{\geq 0}^{6n+1}} \; z - \sum\limits_{\xi\in P\cup Q}\left(\alpha_{\xi}+\beta_{\xi}\right) \notag\\
  &\;\; \;\;\;\;\;\;\;\; \;\;\;\;\;\;\;\; \;\;\;\;\;\;\;-\sum\limits_{p\in P}\gamma_{p, i_Q}-\sum\limits_{q\in Q}\gamma_{q, j_P}    \notag \\
  \textrm{s.t.}&\;\;  \forall p \in P, \;\;z+\textcolor{black}{p_{j_P}}+\omega_{j_P}\textcolor{black}{-p_{i_P}}-\omega_{i_P}\leq \alpha_{p}, \notag\\
                                               &\;\; \;\;\;\;\;\;\;\; \;\;\;\;\;\;\;\;\; \omega_{j_P}-\omega_{i_P}\leq p_{i_P}-p_{j_P} +\beta_{p}, \notag\\
                                      & \;\; \;\;\;\;\;\;\;\; \;\;\;\;\;\;\;\; \; \omega_{i_Q}-\omega_{j_P}\leq p_{j_P}-p_{i_Q}+\gamma_{p, i_Q}, \notag\\
    & \;\; \;\;\;\;\;\;\;\; \;\;\;\;\;\;\;\;\;\;\;\;  \omega_{l}-\omega_{j_P}\leq p_{j_P}-p_{l} \notag \\  
     & \;\; \;\;\;\;\;\;\;\; \;\;\;\;\;\;\;\;\;  \;\;\;\;\;\;\;\;\; \;\;\;\;\;\;\;\;\; \;\;\;\;\;  \;\;\;\;\;   (\forall l \neq i_P, j_P,i_Q),\notag\\
     & \notag\\
&  \forall q \in  Q, \;\;z+\textcolor{black}{q_{j_Q}}+\omega_{j_Q}\textcolor{black}{-q_{i_Q}}-\omega_{i_Q}\leq \alpha_{q}, \notag\\
                        &\;\; \;\;\;\;\;\;\;\; \;\;\;\;\;\;\;\; \omega_{j_Q}-\omega_{i_Q}\leq q_{i_Q}-q_{j_Q} +\beta_{q}, \notag\\
 & \;\; \;\;\;\;\;\;\;\; \;\;\;\;\;\;\;\;   \omega_{j_P}-\omega_{j_Q}\leq q_{j_Q}-q_{j_P}+\gamma_{q, j_P}, \notag\\
& \;\; \;\;\;\;\;\;\;\; \;\;\;\;\;\;\;\;\;\;\; \omega_{l}-\omega_{j_Q}\leq q_{j_Q}-q_{l}  \notag \\
& \;\; \;\;\;\;\;\;\;\; \;\;\;\;\;\;\;\;\;  \;\;\;\;\;\;\;\;\; \;\;\;\;\;\;\;\;\; \;\;\;\;\; \;\;\;\;\;      (\forall l \neq i_P, j_P,i_Q). \notag\\
 &   \tag{{\bf LP 2}}  \label{equation:lp6}
  %
  \end{align}
  }
  \end{minipage}
  
   \begin{minipage}{.5\textwidth}  
  {\scriptsize
    \begin{align}
 & \max \limits_{\left(z; \alpha; \beta\right) \in \mathbb{R}_{\geq 0}^{4n+1}} \; z - \sum\limits_{\xi\in P\cup Q}\left(\alpha_{\xi}+\beta_{\xi}\right)    \notag \\
 &\notag\\
  \textrm{s.t.}\;\; & \forall p \in P, \;\;z+\textcolor{black}{p_{k_2}}+\omega_{k_2}\textcolor{black}{-p_{k_1}}-\omega_{k_1}\leq \alpha_{p}, \notag\\
 &\;\; \;\;\;\;\;\;\;\; \;\;\;\;\;\;\;\; \omega_{k_2}-\omega_{k_1}\leq p_{k_1}-p_{k_2} +\beta_{p}, \notag\\
   &\;\; \;\;\;\;\;\;\;\; \;\;\;\;\;\;\;\;  \;\; \;\omega_{l}-\omega_{k_2}\leq p_{k_2}-p_{l}\notag\\
  & \;\; \;\;\;\;\;\;\;\; \;\;\;\;\;\;\;\;\;  \;\;\;\;\;\;\;\;\; \;\;\;\;\;\;\;\;\; \;\;\;\;\; \;\;\;\;\;   \;\;\; \;\;\;\;     (\forall l \neq k_1, k_2), \notag\\
  & \notag\\
 &  \forall q \in  Q, \;\;z+\textcolor{black}{q_{k_1}}+\omega_{k_1}\textcolor{black}{-q_{k_2}}-\omega_{k_2}\leq \alpha_{q}, \notag\\
  &\;\; \;\;\;\;\;\;\;\; \;\;\;\;\;\;\;\;\omega_{k_1}-\omega_{k_2}\leq q_{k_2}-q_{k_1} +\beta_{q}, \notag\\
  &\;\; \;\;\;\;\;\;\;\; \;\;\;\;\;\;\;\;  \;\; \;\omega_{l}-\omega_{k_1}\leq q_{k_1}-q_{l} \notag\\
  & \;\; \;\;\;\;\;\;\;\; \;\;\;\;\;\;\;\;\;  \;\;\;\;\;\;\;\;\; \;\;\;\;\;\;\;\;\; \;\;\;\;\; \;\;\;\;\;   \;\;\; \;\;\;\;     (\forall l \neq k_1, k_2).  \notag\\
  &
   \tag{{\bf LP 3}} \label{equation:lp3}  
  %
  \end{align}
  }
  \end{minipage}
 \begin{minipage}{.5\textwidth}
  {\scriptsize
   \begin{align} 
  & \max \limits_{\left(z; \alpha; \beta; \gamma \right) \in \mathbb{R}_{\geq 0}^{6n+1}} \; z - \sum\limits_{\xi\in P\cup Q}\left(\alpha_{\xi}+\beta_{\xi}\right) \notag\\
  &\;\; \;\;\;\;\;\;\;\; \;\;\;\;\;\;\;\; \;\;\;\;\;\;\;-\sum\limits_{p\in P}\gamma_{p, i_Q}-\sum\limits_{q\in Q}\gamma_{q, i_P}    \notag \\
  \textrm{s.t.}&\;\;  \forall p \in P, \;\;z+\textcolor{black}{p_{j_P}}+\omega_{j_P}\textcolor{black}{-p_{i_P}}-\omega_{i_P}\leq \alpha_{p}, \notag\\
                                               &\;\; \;\;\;\;\;\;\;\; \;\;\;\;\;\;\;\; \omega_{j_P}-\omega_{i_P}\leq p_{i_P}-p_{j_P} +\beta_{p}, \notag\\
                                                & \;\; \;\;\;\;\;\;\;\; \;\;\;\;\;\;\;\;\; \;\;\omega_{i_Q}-\omega_{j}\leq p_{j}-p_{i_Q}+\gamma_{p, i_Q}, \notag\\
    & \;\; \;\;\;\;\;\;\;\; \;\;\;\;\;\;\;\;\;\;\;\;\; \; \omega_{l}-\omega_{j}\leq p_{j}-p_{l} \notag \\  
     & \;\; \;\;\;\;\;\;\;\; \;\;\;\;\;\;\;\;\;  \;\;\;\;\;\;\;\;\; \;\;\;\;\;\;\;\;\; \;\;\;\;\;  \;\;\;\;\;  \;\;\; (\forall l \neq i_P, j, i_Q),\notag\\
&  \forall q \in  Q, \;\;z+\textcolor{black}{q_{j_Q}}+\omega_{j_Q}\textcolor{black}{-q_{i_Q}}-\omega_{i_Q}\leq \alpha_{q}, \notag\\
                        &\;\; \;\;\;\;\;\;\;\; \;\;\;\;\;\;\;\; \omega_{j_Q}-\omega_{i_Q}\leq q_{i_Q}-q_{j_Q} +\beta_{q}, \notag\\
  &\;\; \;\;\;\;\;\;\;\; \;\;\;\;\;\;\;\;\;\;\; \omega_{i_P}-\omega_{j}\leq q_{j}-q_{i_P}+\gamma_{q, i_P}, \notag\\
& \;\; \;\;\;\;\;\;\;\; \;\;\;\;\;\;\;\;\;\;\;\;\;\; \omega_{l}-\omega_{j}\leq q_{j}-q_{l}  \notag \\
& \;\; \;\;\;\;\;\;\;\; \;\;\;\;\;\;\;\;\;  \;\;\;\;\;\;\;\;\; \;\;\;\;\;\;\;\;\; \;\;\;\;\; \;\;\;\;\;   \;\;\;   (\forall l \neq i_P, j, i_Q).  \tag{{\bf LP 4}} \label{equation:lp5}
  %
  \end{align}
  }
\end{minipage}

  \section{Algorithms}\label{sec:algorithm}

                            In this section, we develop Algorithms \ref{alg:16}--\ref{alg:14} according to the simplified linear programmings \eqref{equation:lp7}--\eqref{equation:lp5} respectively for computing an optimal  tropical hyperplane, which separates two categories of phylogenetic trees.         More details  on the performance of these algorithms can be found   in Section \ref{experiment}.  
                            
        Briefly,   the input of each algorithm includes a pair of  sets $P$ and $Q$  ($P\cap Q=\emptyset$ and $|P|=|Q|=n$), a test set $T$ and indices $i_P, j_P, i_Q, j_Q$ for formulating the corresponding linear programming. 
  As what has been defined in Section \ref{sec:th}, the sets $P$ and $Q$ are associated with a training dataset 
  $\{(x^{(1)}, y_1), \ldots, (x^{(2n)}, y_{2n})\}\subset \R^d \!/ \R \one \times \{0, 1\}$ such that 
 if $y_k=0$, then $x^{(k)}\in P$ and if $y_k=1$, then $x^{(k)}\in Q$.  Here, in these algorithms, we simply call $P$ and $Q$ training sets. 
 The test set $T$ is a finite subset of $\R^d \!/ \R \one$. The indices 
 $i_P, j_P, i_Q, j_Q$ are all from $\{1, \ldots, d\}$ and they satisfy $i_P\neq j_P$, $i_Q\neq j_Q$, and $i_P\neq i_Q$.
   There are two main steps in each algorihm:
        
 \noindent
     {\bf Step 1.}    In Algorithms \ref{alg:16}--\ref{alg:14}, for the input sets $P$ and $Q$ and indices $i_P, j_P, i_Q, j_Q$, we solve the corresponding linear programming (\eqref{equation:lp7}--\eqref{equation:lp5} respectively) and obtain the normal vector $\omega$ of an optimal  tropical hyperplane, which separates the two categories of data $P$ and $Q$. 

 \noindent     
     {\bf Step 2. } After that, for each point $t$ from the test set $T$,  we add $t$ into the set $\tilde{P}$ or $\tilde{Q}$ (that means we classify the point $t$ as  the category $P$ or $Q$) according to which sector of $H_{\omega}$ the point $t$ is located in.   
 As a result, the test set $T$ will be divided into two subsets $\tilde{P}$ and $\tilde{Q}$, and the output of each algorithm is the optimal normal vector $\omega$ and 
 a  partition of  the test set:  $\tilde{P}$ and $\tilde{Q}$.

Below, we give more details for the above Step 2.  The key of this step is to decide which category a  point $t$ from the test set should go once we have the optimal $H_{\omega}$ solved from the linear programming.  The point $t$ might be located  in a closed sector,  or on an intersection of many different closed sectors of $H_{\omega}$
 (in comparison, for a soft margin classical  SVM, a point from the test set might be simply in one of two open half-spaces  determined by an optimal hyperplane, or on the hyperplane). 
 Remark that for a tropical hyperplane in $\mathbb R^d \!/\mathbb R {\bf
    1}$,  there are $d$ closed sectors and $2^d-d-1$ possible intersections of different closed sectors.   
 So far, we do not prove any criteria on how to classify a point according to its  location. Here, according to substantial experiments on simulated data  generated by the multispecies coalescent process,  we propose an effective  strategy in Algorithms \ref{alg:16}--\ref{alg:14} as follows. Since our input training data $P$ and $Q$ are generated by  the  multispecies coalescent process,
 we also input two numbers $C$ and $\eta$  to these algorithms, where $C$ denotes the ratio of the depth of the species tree to the effective population size in the  multispecies coalescent process (see \eqref{Cvalue} in Section 
 \ref{experiment}), and $\eta$ is a threshold (experiments show that a real number between $4$ and $5$ is a good choice for $\eta$). 
In each algorithm,  for the input data $P$ and $Q$, we provide two ways to classify a point from the test set according to the relative values of $C$ and $\eta$. 
For instance, in Algorithm \ref{alg:15},  the variable ${\mathcal I^*}$ in  Line \ref{Istar3} has $4$ possible values: 
$\{k_1, k_2\}$, $\{k_1\}$, $\{k_2\}$, or $\emptyset$.  That respectively means the current point $t\in T$ read by Line \ref{loop3} is located on the intersection of 
$\overline{S}_{\omega}^{k_1}$ and $\overline{S}_{\omega}^{k_2}$, in  the difference 
 $\overline{S}_{\omega}^{k_1}\backslash \overline{S}_{\omega}^{k_2}$, in the difference  $\overline{S}_{\omega}^{k_2}\backslash \overline{S}_{\omega}^{k_1}$, or other cases. 
When the input $C$ is not larger than the input $\eta$, we apply one method to classify $t$ (see Lines \ref{smallcp3}--\ref{smallcq3}), and when $C$ is larger than $\eta$, we apply another method (see Lines \ref{largecp3}--\ref{largecq3}). The other three algorithms are similarly designed. 
 Experimental results show that our algorithms give good accuracy rate for random simulated data (see Figure \ref{fig:accuracy} in Section \ref{experiment}). 
                            

      
           \begin{algorithm}[t]
                                       \tiny
                                       \DontPrintSemicolon
                                       \LinesNumbered
                                       \SetKwInOut{Input}{input}
                                       \SetKwInOut{Output}{output}
                                       \Input{Training sets: $P$, $Q$; ~~Test set: $T$; ~~Indices: pairwise distinct $i_P$, $i_Q$, $j_P$, $j_Q$;~~Threshold: $\eta>0$;~~Parameter: $C>0$
                                          (for the multispecies coalescent process)
                                       }
                                       \Output{               
                                    Optimal normal vector: $\omega$;~~
                                    A partition of $T$: $\tilde P$, $\tilde Q$ such that $\tilde P\cup \tilde Q=T$ and $\tilde P\cap \tilde Q=\emptyset$ 
                                                           }
                                       Solve the linear programming \eqref{equation:lp7} for input data sets $P$, $Q$ and indices $i_P$, $i_Q$, $j_P$, $j_Q$\; 
                                       $\omega \leftarrow$ optimal $\omega$ such that  the objective function in  \eqref{equation:lp7}  reaches its optimal value\;
                                       $\tilde P\leftarrow \emptyset$, ~~ $\tilde Q\leftarrow \emptyset$\;
                                       \For{each $t \in T$}
                                              {
                                                   ${\mathcal I}\leftarrow$ the set of indices  $\{i|\omega_i+t_i=\max \{\omega_k+t_k| 1\leq k\leq d\}, 1\leq i\leq d\}$\;
                                                  ${\mathcal I}^*\leftarrow$ ${\mathcal I}\cap \{i_P,  i_Q, j_P, j_Q\}$\;
                                                      \If{$C\leq \eta$ \nllabel{loop1}}{
                                                  {\bf if} {${\mathcal I}^*=\{i_P\}, \{j_P\}, \{i_P, j_P\}, \{i_P, i_Q\}, \{j_P, i_Q\}, \{j_P, j_Q\}, \{j_P, i_Q, j_Q\}$, or $\{i_P, j_P, i_Q, j_Q\}$} {\bf then} {Add $t$ into $\tilde P$}\;
                                                    {\bf if} {${\mathcal I}^*=\{i_Q\}, \{j_Q\}, \{i_Q, j_Q\}, \{i_P, j_Q\}, \{i_P, j_P, i_Q\}, \{i_P, j_P, j_Q\}, \{i_P, i_Q, j_Q\}$, or $\emptyset$} {\bf then} {Add $t$ into $\tilde Q$}\; 
                                                    }
                                                    \If{$C> \eta$}{  
                                                      {\bf if} {${\mathcal I}^*=\{i_P\}, \{j_P\}, \{i_P, j_P\}, \{i_P, i_Q\},  \{j_P, j_Q\}, \{i_P, i_Q, j_Q\}, \{j_P, i_Q, j_Q\}$, or $\{i_P, j_P, i_Q, j_Q\}$} {\bf then} {Add $t$ into $\tilde P$}\;
                                                    {\bf if} {${\mathcal I}^*=\{i_Q\}, \{j_Q\}, \{i_Q, j_Q\}, \{i_P, j_Q\}, \{j_P, i_Q\}, \{i_P, j_P, i_Q\}, \{i_P, j_P, j_Q\}$, or $\emptyset$ } {\bf then} {Add $t$ into $\tilde Q$\nllabel{loop1end}}\; 
                                                                                                            }

                                               } 
                                        {\bf return} $\omega$, $\tilde P$, $\tilde Q$        
                                      
                                       \caption{{\bf Tropical Classifier via LP 1}}\label{alg:16}
                                       \end{algorithm}

         \begin{algorithm}[t]
                                       \tiny
                                       \DontPrintSemicolon
                                       \LinesNumbered
                                       \SetKwInOut{Input}{input}
                                       \SetKwInOut{Output}{output}
                                    \Input{Training sets: $P$, $Q$; ~~Test set: $T$; ~~Indices: pairwise distinct $i_P$, $i_Q$, $j_P$ ($j_Q=i_P$);~~
                                    Threshold: $\eta>0$;~~Parameter: $C>0$
                                          (for the multispecies coalescent process)
                                      }
                                       \Output{               
                                    Optimal  normal vector: $\omega$;~~
                                    A partition of $T$: $\tilde P$, $\tilde Q$ such that $\tilde P\cup \tilde Q=T$ and $\tilde P\cap \tilde Q=\emptyset$ 
                                                         }                 
                                        Solve the linear programming \eqref{equation:lp6} for input data sets $P$, $Q$ and indices $i_P$, $i_Q$, $j_P$\; 
                                       $\omega \leftarrow$ optimal $\omega$ such that  the objective function in  \eqref{equation:lp6}  reaches its optimal value\;  
                                         $\tilde P\leftarrow \emptyset$, ~~ $\tilde Q\leftarrow \emptyset$\;
                                       \For{each $t \in T$}
                                              {
                                              ${\mathcal I}\leftarrow$ the set of indices  $\{i|\omega_i+t_i=\max \{\omega_k+t_k| 1\leq k\leq d\}, 1\leq i\leq d\}$\;
                                                  ${\mathcal I}^*\leftarrow$  ${\mathcal I}\cap \{i_P,  i_Q, j_P\}$\;
                                                      \If{$C\leq \eta$}{
                                                  {\bf if} {${\mathcal I}^*=\{i_P\}, \{j_P\}, \{i_Q\}$, or $\{i_Q, j_P\}$} {\bf then} {Add $t$ into $\tilde P$}\;
                                                    {\bf if} {${\mathcal I}^*=\{i_P, j_P\}, \{i_P, i_Q\}, \{i_P, i_Q, j_P\}$, or $\emptyset$} {\bf then} {Add $t$ into $\tilde Q$}\; 
                                                    }
                                                    \If{$C> \eta$}{  
                                                      {\bf if} {${\mathcal I}^*=\{i_P\}, \{i_P, j_P\}, \{i_P, i_Q, j_P\}$, or $\emptyset$} {\bf then} {Add $t$ into $\tilde P$}\;
                                                    {\bf if} {${\mathcal I}^*=\{j_P\}, \{i_Q\}, \{i_P, i_Q\}$, or $\{i_Q, j_P\}$} {\bf then} {Add $t$ into $\tilde Q$}\; 
                                                                                                            }
                                               
                                               } 
                                        {\bf return} $\omega$, $\tilde P$, $\tilde Q$       
                                       \caption{{\bf Tropical Classifier via LP 2}}\label{alg:17}
                                       \end{algorithm}

                             
                                   \begin{algorithm}[t]
                                       \tiny
                                       \DontPrintSemicolon
                                       \LinesNumbered
                                       \SetKwInOut{Input}{input}
                                       \SetKwInOut{Output}{output}
                                          \Input{Training sets: $P$, $Q$; ~~Test set: $T$; ~~Indices:
                                     distinct  $k_1$, $k_2$ ($i_P=j_Q=k_1$, $i_Q=j_P=k_2$); ~~Threshold: $\eta>0$;~~Parameter: $C>0$
                                          (for the multispecies coalescent process)}
                                       \Output{               
                                    Optimal  normal vector: $\omega$;~~
                                    A partition of $T$: $\tilde P$, $\tilde Q$ such that $\tilde P\cup \tilde Q=T$ and $\tilde P\cap \tilde Q=\emptyset$ 
                                                           }
                                       Solve the linear programming \eqref{equation:lp3} for input data sets $P$, $Q$ and indices $k_1$, $k_2$\; 
                                       $\omega \leftarrow$ optimal $\omega$ such that  the objective function in  \eqref{equation:lp3}  reaches its optimal value\;
                                         $\tilde P\leftarrow \emptyset$, ~~ $\tilde Q\leftarrow \emptyset$\;
                                       \For{each $t \in T$\nllabel{loop3}}
                                              {
                                               ${\mathcal I}\leftarrow$ the set of indices  $\{i|\omega_i+t_i=\max \{\omega_k+t_k| 1\leq k\leq d\}, 1\leq i\leq d\}$\;
                                                  ${\mathcal I}^*\leftarrow$  ${\mathcal I}\cap \{k_1, k_2\}$\nllabel{Istar3}\;
                                                     \If{$C\leq \eta$}{
                                                  {\bf if} {${\mathcal I}^*=\{k_1\}$, or $\{k_1, k_2\}$} {\bf then} {Add $t$ into $\tilde P$}\nllabel{smallcp3}\;
                                                    {\bf if} {${\mathcal I}^*=\{k_2\}$, or $\emptyset$} {\bf then} {Add $t$ into $\tilde Q$}\nllabel{smallcq3}\;
                                                    }
                                                      \If{$C> \eta$}{
                                                  {\bf if} {${\mathcal I}^*=\{k_1\}$, or $\emptyset$} {\bf then} {Add $t$ into $\tilde P$}\nllabel{largecp3}\;
                                                    {\bf if} {${\mathcal I}^*=\{k_2\}$, or $\{k_1, k_2\}$} {\bf then} {Add $t$ into $\tilde Q$}\nllabel{largecq3}\;
                                                    }
                                               } 
                                        {\bf return} $\omega$, $\tilde P$, $\tilde Q$       
                                       \caption{{\bf Tropical Classifier via LP 3}}\label{alg:15}
                                       \end{algorithm}
                                       
                                          \begin{algorithm}[h]
                                       \tiny
                                       \DontPrintSemicolon
                                       \LinesNumbered
                                       \SetKwInOut{Input}{input}
                                       \SetKwInOut{Output}{output}
                                          \Input{Training sets: $P$, $Q$; ~~Test set: $T$; ~~Indices: pairwise distinct $i_P$, $i_Q$, $j$ ($j_P=j_Q=j$); ~~Threshold: $\eta>0$;~~Parameter: $C>0$
                                          (for the multispecies coalescent process)
                                      }
                                       \Output{               
                                    Optimal  normal vector: $\omega$;~~
                                    A partition of $T$: $\tilde P$, $\tilde Q$ such that $\tilde P\cup \tilde Q=T$ and $\tilde P\cap \tilde Q=\emptyset$ 
                                                           }
                                                             Solve the linear programming \eqref{equation:lp5} for input data sets $P$, $Q$ and indices $i_P$, $i_Q$, $j$\; 
                                       $\omega \leftarrow$ optimal $\omega$ such that  the objective function in  \eqref{equation:lp5}  reaches its optimal value\; 
                                         $\tilde P\leftarrow \emptyset$, ~~ $\tilde Q\leftarrow \emptyset$\;
                                       \For{each $t \in T$}
                                              {
                                                 ${\mathcal I}\leftarrow$ the set of indices  $\{i|\omega_i+t_i=\max \{\omega_k+t_k| 1\leq k\leq d\}, 1\leq i\leq d\}$\;
                                                  ${\mathcal I}^*\leftarrow$  ${\mathcal I}\cap \{i_P, i_Q, j\}$\;
                                                      \If{$C\leq \eta$}{
                                                  {\bf if} {${\mathcal I}^*=\{i_Q\}, \{i_P, i_Q\}, \{i_P, j\}$, or $\{j\}$} {\bf then} {Add $t$ into $\tilde P$}\;
                                                    {\bf if} {${\mathcal I}^*=\{i_P\}, \{i_Q, j\},  \{i_P, i_Q, j\}$, or $\emptyset$} {\bf then} {Add $t$ into $\tilde Q$}\; 
                                                    }
                                                    \If{$C> \eta$}{  
                                                      {\bf if} {${\mathcal I}^*=\{i_P\}, \{i_P, i_Q\}, \{i_P, i_Q, j\}$, or $\emptyset$} {\bf then} {Add $t$ into $\tilde P$}\;
                                                    {\bf if} {${\mathcal I}^*=\{i_Q\}, \{i_P, j\}, \{i_Q, j\}$, or $\{j\}$} {\bf then} {Add $t$ into $\tilde Q$}\; 
                                                                                                            }
                                               } 
                                        {\bf return} $\omega$, $\tilde P$, $\tilde Q$   
                                       
                                       \caption{{\bf Tropical Classifier via LP 4}}\label{alg:14}
                                       \end{algorithm}

\section{Implementation and Experiment}\label{experiment} 
In this section, we apply Algorithms \ref{alg:16}--\ref{alg:14} implemented in {\tt R} (see "Algorithm 1. R"-"Algorithm 4. R" in Table \ref{table:files}) and a classical SVM in the {\tt R}  package {\tt e1071} \cite{classicalSVM} to simulated data sets of gene trees generated by the multispecies coalescent process.  

First, we describe how we generated the simulated data. We set the parameters for the multispecies coalescent process, species depth ($Depth$) under the multispecies coalescent model  and effective population size ($Population$) as:
  \begin{equation}\label{Cvalue}
{Depth} \;= \;{Population} \times C,
\end{equation}
where $C$ is a constant, which takes a given value from $$\{0.2, 0.4, 0.6, 0.8,1, 1.2, 2.4, 3.6, 4.8, 6, 8, 10\}.$$ 
In this simulation studies, we fixed $Population = 10,000$ and determine $Depth$ by Equation \eqref{Cvalue}.
For each value of $C$, we generate a pair of  simulated data sets $P$ and $Q$ by the following steps:
  \begin{enumerate}
\item\label{step1} We generate $100$ species trees with $5$ leaves under the Yule process by {\tt Mesquite} \cite{mesquite}. 
\item\label{step2}  Given a species tree generated in Step \ref{step1}, we  generate $100$ gene trees with $5$ leaves each under the coalescent model. We set $P$ a set of these $100$ gene trees. 
\item\label{step3} By repeating Step \ref{step2} with another species tree generated in Step \ref{step1}, we obtain another set of $100$ gene trees and we set this set of gene trees   $Q$. 
\item\label{step4}  We convert gene trees in $P$ generated by Step \ref{step2} and gene trees in $Q$ generated by Step \ref{step3} into ultrametrics of length $\binom{5}{2}=10$ by the {\tt R} package {\tt ape} \cite{ape}.
\end{enumerate}
  \begin{figure}[h]
\centering
\includegraphics[width=0.32\textwidth, height=0.3\textwidth]{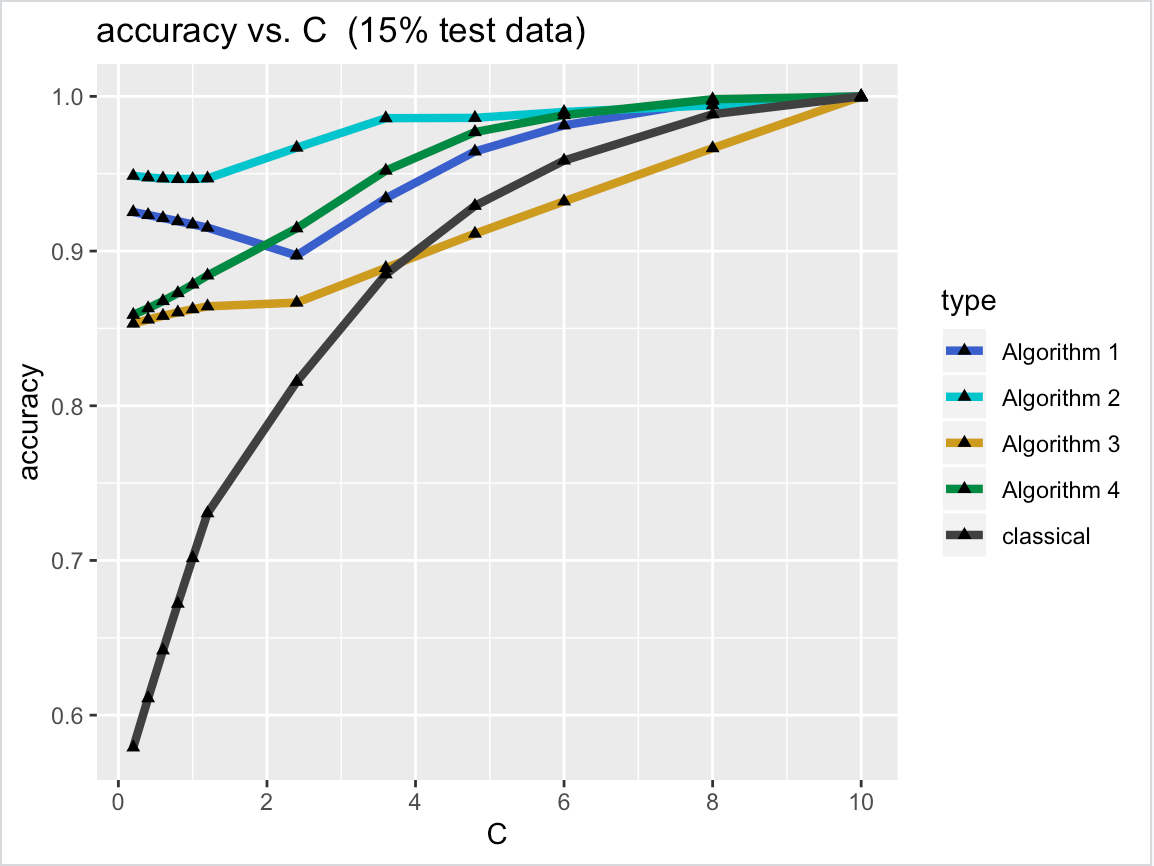}
\includegraphics[width=0.32\textwidth, height=0.3\textwidth]{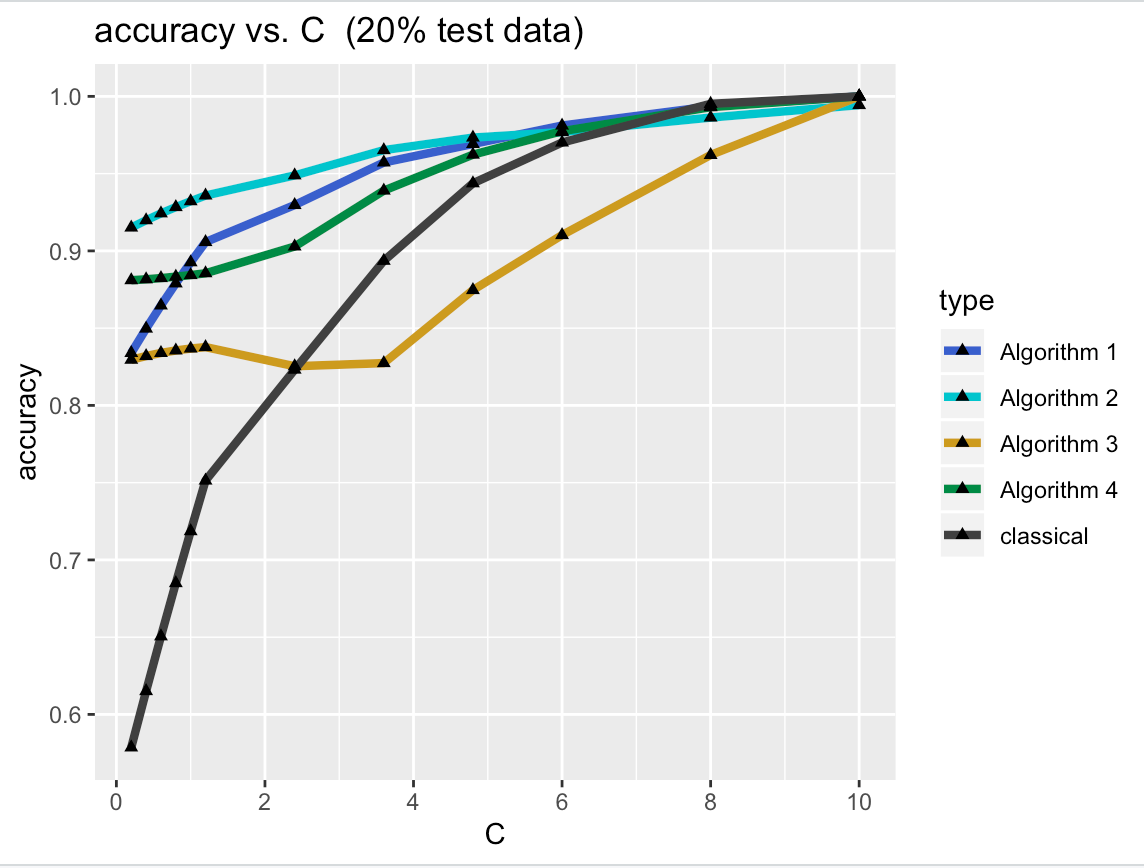}
\includegraphics[width=0.32\textwidth, height=0.3\textwidth]{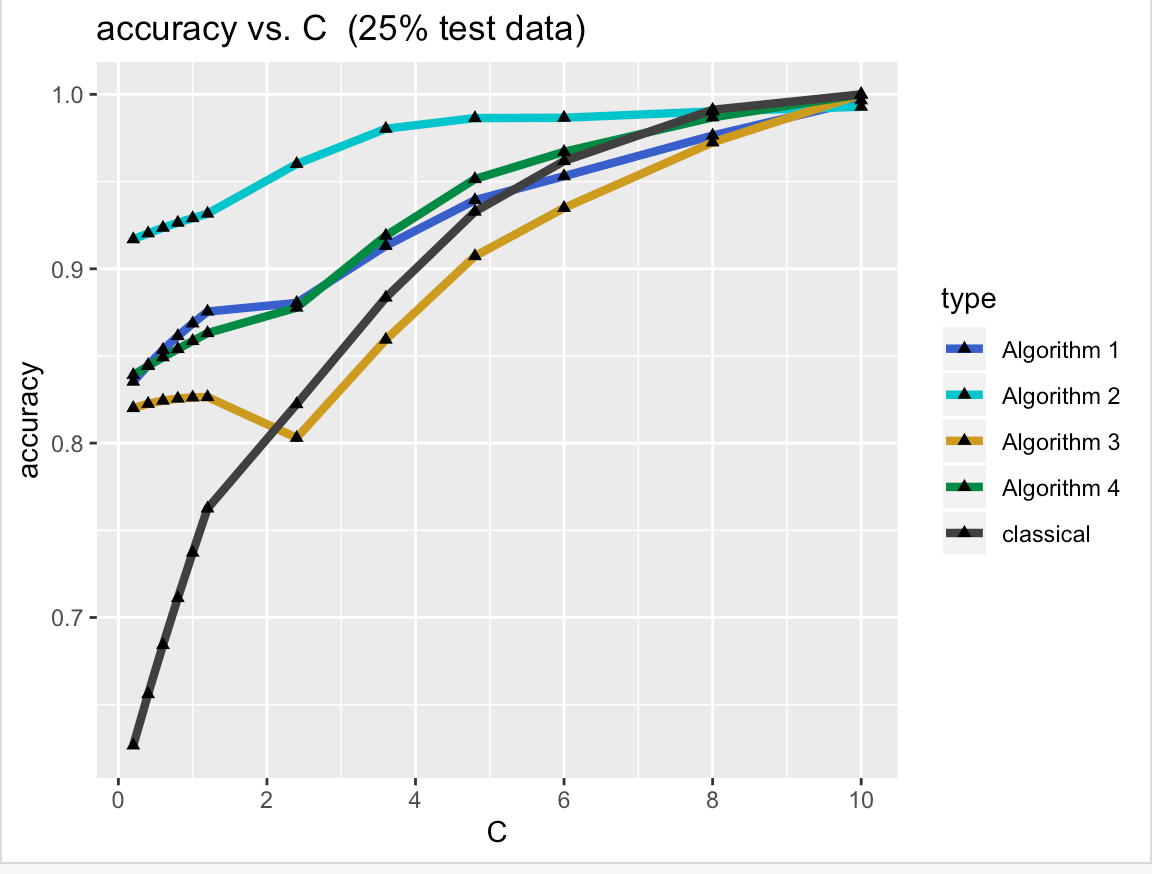}
\caption{Accuracy rates for different proportions of points: $15\%$, $20\%$ and $25\%$, respectively. The y-axis represents the accuracy rates. 
   All of them are based
  on simulated data sets generated under the multispecies coalescent processes via {\tt Mesquite}.  The x-axis shows the ratios of the species depth to effective population size of the coalescent processes. 
  We used the {\tt ggplot2} package \cite{ggplot} to plot them. All lines were smoothed by the Lowess function in {\tt R}.}
\label{fig:accuracy}
\end{figure}

We applied (soft margin)  tropical SVMs  and classical SVMs  to the  simulated data sets generated by the procedure described above and computational results are shown in Figure \ref{fig:accuracy}. 
For each pair of sets $P$ and $Q$, we chose 
a proportion ($15\%$, $20\%$, or $25\%$) of random points from $P$ and $Q$ respectively as our test set, and the rest points in $P$ and $Q$ form a training dataset.   
To obtain the curve marked as ``classical" in Figure \ref{fig:accuracy}, we applied  classical SVMs to the simulated data sets and we recorded their accuracy rates. To obtain the curves marked as ``Alogirhm \ref{alg:16}"-``Algorithm \ref{alg:14}" in Figure \ref{fig:accuracy}, for each training dataset (according to its related $C$), 
we applied Algorithms \ref{alg:16} to \ref{alg:14} respectively to obtain classifications with the test set, where the input threshold $\eta$ is $4.8$. 
Note that for each $C$ and for each proportion ($15\%$, $20\%$, or $25\%$), we randomly sampled the test sets $10$ different times, and we recorded the best accuracy rate among the $10$ times of computation. Figure \ref{fig:accuracy} can be produced by running "Graph Producer.R" (see Table \ref{table:files}). 
The 
input training sets, test set and indices ($i_P, j_P, i_Q, j_Q$)  of these algorithms can be found online (see the folders "Genetree Data" and "Assignment" in Table \ref{table:files}).

Figure \ref{fig:accuracy} shows that when the value of $C$ is small, tropical SVMs give much higher accuracy rates than that of classical SVMs. More specifically,  in all three figures, all
Algorithms \ref{alg:16}--\ref{alg:14} have much better accuracy rates than that of classical SVMs when 
$C$ is between $0$ and $2$. And Algorithm \ref{alg:17}  has the best  accuracy rate when 
$C$ is less than $6$. Algorithm \ref{alg:15} does not behave as well as the other three Algorithms. A possible reason is that it only uses two sectors (indexed by $\{k_1, k_2\}$) to classify the training dataset, while 
Algorithm \ref{alg:16}, Algorithm \ref{alg:17} and Algorithm \ref{alg:14} use four, three and three sectors (indexed by $\{i_P, i_Q, j_P, j_Q\}$, $\{i_P, j_P, i_Q\}$ and $\{i_P, i_Q, j\}$), respectively.  
One may ask why Algorithm \ref{alg:16} does not behave better than 
Algorithm \ref{alg:17} even though Algorithm \ref{alg:16} uses more sectors. Our explanation is that the method of classifying test points presented in Algorithm \ref{alg:16} (Lines \ref{loop1}-\ref{loop1end}) might not be the best way for the corresponding linear programming problem \eqref{equation:lp7}.  

  From Figure \ref{fig:accuracy}, we observed that accuracy rates tend to improve as $C$ increased in overall.  This can be explained by the nature of the multispecies coalescent process.  When $C$    is small, then a gene tree generated by the coalescent process is not constrained by the tree topology of the species tree so that a gene tree becomes a random tree. Therefore, the two sets $P$ and $Q$ of gene trees tend to be overlapped over the space of ultrametrics.   When $C$ is large, then the tree topologies of gene trees under a fixed species tree tend to be strongly correlated.  Therefore, the two sets $P$ and $Q$ of gene trees tend to be well separated in the space of ultrametrics. 

  Also, in order to show the performance of algorithms in Section \ref{sec:algorithm}, we recorded the computation time in seconds for each algorithm and classical SVMs shown in Table \ref{table:timing}. 
  Each computational time in Table \ref{table:timing} is the average time over $10$ times of computation for the same input.  Here,
   the input training data  sets $P$ and $Q$ have 
  $170$ gene trees with $5$ leaves in total, and the test set $T$ have $30$ gene trees. These data sets are generated by the same 
  multispecies coalescent process (for $C=10$ and $Population = 10,000$) as before. 
  
\begin{table}[h!]
\centering
\caption{Computational time (s: second) of each method.}\label{table:timing}
\begin{tabular}{||c c c c c||} 
\hline
Algorithm \ref{alg:16}  &  Algorithm \ref{alg:17} & Algorithm \ref{alg:15}  &  Algorithm \ref{alg:14}  & Classical \\ [0.5ex] 
\hline\hline
$0.052s$ & $0.033s$ & $0.012s$ & $0.037s$ & $0.0041s$  \\[1ex]
\hline
\end{tabular}
\end{table}

\section{Discussion}\label{sec:discussion}
Here we focused on a tropical SVM, a tropical hyperplane over the tropical projective torus, which separates data points from different categories into  sectors. We formulated both hard margin and soft margin tropical SVMs as linear programming problems.  For tropical SVMs,  we need to go through exponentially many linear programming problems in terms of the dimension of the tropical projective torus and the sample size of the input data to separate data points.  Therefore, in this paper, we explored a simper case when 
all points from the same category are staying in the same sector of a separating tropical hyperplane. 
For the hard margin tropical SVMs, we proved the necessary and sufficient conditions to separate two tropically separable sets and  we showed the explicit formula of an optimal solution for a feasible linear programming problem.  For soft margin tropical SVMs, we simplified the linear programming problems by studying their properties, and    
 we developed algorithms to compute a soft margin tropical SVM from data in the tropical projective torus. We compared our methods (implemented in {\tt R}) with the svm() funciton from the {\tt e1071} package \cite{classicalSVM} for the simulated data generated by the multispecies coalescent processes \cite{coalescent}.  

In general we have to go through exponentially many linear programming problems to find soft margin  and hard margin  tropical SVMs in our algorithms. However, we do not know exactly the time complexity of a tropical SVM.  
\begin{problem}
What is the time complexity of a hard margin tropical SVM over the tropical projective torus? How about the time complexity of a soft margin tropical SVM over the tropical projective torus?  Is it a NP-hard problem? 
\end{problem}

In addition, in this paper, we focused on tropical SVMs over the tropical projective torus $\mathbb R^{N \choose 2} \!/\mathbb R {\bf 1}$ not over the space of ultrametrics.  Recall that the space of ultrametrics is an union of $N - 2$ dimensional polyhedral cones in $\mathbb R^{N \choose 2} \!/\mathbb R {\bf 1}$.
In our simulation studies, we generated points using the multispecies coalescent model.  These trees are equidistant trees and then they are converted to ultrametrics in the space of ultrametrics.   Our algorithms return tropical SVMs with dimension ${N \choose 2} - 2$ over the tropical projective torus $\mathbb R^{N \choose 2} \!/\mathbb R {\bf 1}$ and then we used them to separate points within the space of ultrametrics, the union of $N-2$ dimensional polyhedral cones subset of $\mathbb R^{N \choose 2} \!/\mathbb R {\bf 1}$. 
In the previous section, our simulations showed tropical SVMs over $\mathbb R^{N \choose 2} \!/\mathbb R {\bf 1}$ worked very well to separate points living inside of the union of $N-2$ dimensional polyhedral cones while the classical SVMs did not work well.  We think there must be some geometrical explanations why tropical SVMs worked well.  
\begin{problem}
Can we describe a tropical SVM over the tropical projective torus $\mathbb R^{N \choose 2} \!/\mathbb R {\bf 1}$ geometrically how it separates points in the space of ultrametrics, the union of $N-2$ dimensional polyhedral cones?  
\end{problem}

Also  we are interested in developing algorithms to compute hard margin tropical and soft margin SVMs over the space of ultrametrics, the union of $N-2$ dimensional polyhedral cones, subset of the tropical projective torus $\mathbb R^{N \choose 2} \!/\mathbb R {\bf 1}$. Since the space of ultrametrics is a tropical linear space over the tropical projective torus $\mathbb R^{N \choose 2} \!/\mathbb R {\bf 1}$,  tropical SVMs should be a tropical linear space with the dimension $N-2$.  As Yoshida et. al in \cite{YZZ} defined a tropical principal components over the space of ultrametrics as vertices of a tropical polytope over the space of ultrametrics and Page et. al showed properties of tropical PCAs over the space of ultrametrics in \cite{page2019tropical}, we might be able to define a tropical SVM over the space of ultrametrics as a tropical polytope over the space of ultrametrics.  
\begin{problem}
Can we define hard and soft margin tropical SVMs over the space of ultrametrics?  If so then how can we formulate them as an optimization problem?
\end{problem}


\begin{appendix}
\section{Technical Details}\label{appA}
\subsection{Proof of Theorem \ref{thm:16}}
\begin{proof}
For any $p\in P$, by the assumptions $i(p)=i_P$ and $j(p)=j_P$, and by \eqref{equation:253}--\eqref{equation:254}, we have:
$$
\begin{array}{cccc}
    \omega_{j_P}-\omega_{i_{P}}\leq p_{i_{P}}-p_{j_P}, &  \text{and}\;\;  \omega_{l} -\omega_{j_P}\leq p_{j_P}-p_{l},\;\; \forall l\neq i_P, j_P.
  \end{array}
$$
By the assumption that the four numbers $i_P, j_P, i_Q$ and $j_Q$ are pairwise distinct, 
there exists $l\neq i_{P},j_P$ such that $l=i_{Q}$ or $l=j_Q$.  Similarly, for any $q \in Q$, by \eqref{equation:253}--\eqref{equation:254},  we have
$$
\begin{array}{cccc}
   \omega_{j_Q}-\omega_{i_{Q}}\leq  q_{i_{Q}}-q_{j_Q}, &  \text{and}\;\;
 \omega_{l}-\omega_{j_Q}\leq    q_{j_Q} - q_{l}, \;\;\forall l\neq i_Q, j_Q, 
  \end{array}
$$
and there exists $l\neq i_{Q}, j_Q$ such that $l=i_P$ or $l=j_{P}$.
So for any $p\in P$, we have

\noindent\begin{tabularx}{\textwidth}{@{}XXX@{}}
\begin{equation}\label{proof:thm1611}
\omega_{j_P}-\omega_{i_{P}}\leq p_{i_{P}}-p_{j_P}, 
\end{equation}&
\begin{equation}\label{proof:thm1612}
\omega_{i_Q} -\omega_{j_P}\leq p_{j_P}-p_{i_Q}, 
\end{equation}&
\begin{equation}\label{proof:thm1613}
\omega_{j_Q}-\omega_{j_P} \leq  p_{j_P}-p_{j_Q}.
\end{equation}
\end{tabularx}
and for any $q\in Q$, 

\noindent\begin{tabularx}{\textwidth}{@{}XXX@{}}
\begin{equation}\label{proof:thm1621}
\omega_{j_Q}-\omega_{i_{Q}}\leq  q_{i_{Q}}-q_{j_Q},
\end{equation}&
\begin{equation}\label{proof:thm1622}
 \omega_{i_P}-\omega_{j_Q}\leq    q_{j_Q} - q_{i_P}, 
\end{equation}&
\begin{equation}\label{proof:thm1623}
 \omega_{j_P}-\omega_{j_Q} \leq  q_{j_Q}-q_{j_P}.
\end{equation}
\end{tabularx}
By adding  \eqref{proof:thm1612} and  \eqref{proof:thm1621}, we have 
\begin{equation}\label{proof:thm165}
\omega_{j_{Q}}-\omega_{j_{P}}\; \leq \;\min\limits_{q\in Q}\{q_{i_{Q}}-q_{j_{Q}}\}+\min\limits_{p\in P}\{p_{j_{P}}-p_{i_{Q}}\}. 
\end{equation}
By adding \eqref{proof:thm1611} and  \eqref{proof:thm1622},  we have 
\begin{align}\label{proof:thm166}
\max_{p\in P}\{p_{j_{P}}-p_{i_{P}}\} + \max_{q\in Q}\{q_{i_{P}}-q_{j_{Q}}\} \; \leq \;\omega_{j_{Q}}-\omega_{j_{P}}.
\end{align}
By \eqref{proof:thm1613}, \eqref{proof:thm1623},  \eqref{proof:thm165} and \eqref{proof:thm166}, the inequality \eqref{eq:thm163} holds. Therefore, if 
the linear programming \eqref{equation:251}--\eqref{equation:254} has a feasible solution, then 
we have  
\eqref{eq:thm163}.

On the other hand, if we have \eqref{eq:thm163}, then
there exist real numbers $\omega_{j_Q}$ and $\omega_{j_P}$ such that the inequalities \eqref{proof:thm1613}, \eqref{proof:thm1623}, 
\eqref{proof:thm165} and  \eqref{proof:thm166} hold. 
Notice that  the  inequality \eqref{proof:thm165} is equivalent to 
\begin{align*}
\omega_{j_{Q}} - \min\limits_{q\in Q}\{q_{i_{Q}}-q_{j_{Q}}\} \leq  \omega_{j_{P}}+\min\limits_{p\in P}\{p_{j_{P}}-p_{i_{Q}}\} .
\end{align*}
So, there exists a number $\omega_{i_Q}$ such that 

\noindent\begin{tabularx}{\textwidth}{@{}XXX@{}}
\begin{equation}\label{proof:thm1691}
\omega_{j_{Q}} - \min\limits_{q\in Q}\{q_{i_{Q}}-q_{j_{Q}}\} \leq \omega_{i_Q}, 
\end{equation}&
\begin{equation}\label{proof:thm1692}
\omega_{i_Q}\leq\omega_{j_{P}}+\min\limits_{p\in P}\{p_{j_{P}}-p_{i_{Q}}\} .
\end{equation}
\end{tabularx}
Symmetrically,  the inequality  \eqref{proof:thm166} is equivalent to 
\begin{align*}
 \max_{p\in P}\{p_{j_{P}}-p_{i_{P}}\}  + \omega_{j_{P}}\leq \omega_{j_{Q}} - \max_{q\in Q}\{q_{i_{P}}-q_{j_{Q}}\}.
\end{align*}
So, there exists a number $\omega_{i_P}$ such that 

\noindent\begin{tabularx}{\textwidth}{@{}XXX@{}}
\begin{equation}\label{proof:thm1671}
 \max_{p\in P}\{p_{j_{P}}-p_{i_{P}}\}  + \omega_{j_{P}}\leq  \omega_{i_P},
\end{equation}&
\begin{equation}\label{proof:thm1672}
\omega_{i_P}\leq \omega_{j_{Q}} - \max_{q\in Q}\{q_{i_{P}}-q_{j_{Q}}\}.
\end{equation}
\end{tabularx}
The  inequality \eqref{proof:thm1671} can be rewritten as 
\begin{align}\label{proof:thm168}
 \max_{p\in P}\{p_{j_{P}}-p_{i_{P}}\}  \leq  \omega_{i_P}-\omega_{j_{P}} \;\Leftrightarrow \; \omega_{j_P}-\omega_{i_{P}} \leq \min\limits_{p\in P}\{p_{i_{P}}-p_{j_{P}}\}  .
 \end{align}
The  inequality  \eqref{proof:thm1672} can be rewritten as 
\begin{align}\label{proof:thm1610}
 \omega_{i_{P}}-\omega_{j_{Q}} \leq - \max_{q\in Q}\{q_{i_{P}}-q_{j_{Q}}\}
 \;\Leftrightarrow \; \omega_{i_P}-\omega_{j_{Q}} \leq \min\limits_{q\in Q}\{q_{j_{Q}}-q_{i_{P}}\}  .
 \end{align}
 By \eqref{proof:thm1691} and  \eqref{proof:thm168}, the inequality 
 \eqref{equation:253} holds. 
 By \eqref{proof:thm1613}, \eqref{proof:thm1623},   \eqref{proof:thm1692} and \eqref{proof:thm1610},  
the inequality \eqref{equation:254} holds for 
$l= i_Q, j_Q$ when $\xi\in P$, or for $l=i_P, j_P$ when $\xi\in Q$. 
For $l\neq i_Q, j_Q$ when $\xi\in P$, or for $l\neq i_P,  j_P$ when $\xi\in Q$, there always exist sufficiently small numbers for  $\omega_{l}$ such that the inequality \eqref{equation:254} holds. So the inequality \eqref{eq:thm163} guarantees the feasibility of the inequalities \eqref{equation:253} and \eqref{equation:254}. Notice that once \eqref{equation:253} and \eqref{equation:254} are feasible, there is always a non-negative number $z$ such that
\eqref{equation:252} holds. So, if we have \eqref{eq:thm163}, then the linear programming \eqref{equation:251}--\eqref{equation:254} has a feasible solution.

If a feasible solution exists, then by \eqref{equation:252}, for any feasible solution $(\omega; z)$, 
 \begin{align}\label{proof:thm16100}
    z\leq \min\limits_{p \in P}\{p_{i_P}-p_{j_P}\} + \omega_{i_P}-\omega_{j_P}, \;\; z\leq \min\limits_{q \in Q}\{q_{i_Q}-q_{j_Q}\} + \omega_{i_Q}-\omega_{j_Q}.
\end{align}
So,  by \eqref{proof:thm1612} and \eqref{proof:thm1622}, and by summing up the above two inequalities, we have 
\begin{align}
2z &\;\leq \;\min\limits_{p \in P}\{p_{i_P}-p_{j_P}\} + \min\limits_{q \in Q}\{q_{i_Q}-q_{j_Q}\} + \omega_{i_P}-\omega_{j_P}+ \omega_{i_Q}-\omega_{j_Q}  \label{proof:thm16101}\\
    &\; = \;\min\limits_{p \in P}\{p_{i_P}-p_{j_P}\} + \min\limits_{q \in Q}\{q_{i_Q}-q_{j_Q}\} + (\omega_{i_Q}-\omega_{j_P})+ (\omega_{i_P}-\omega_{j_Q}) \notag \\
    &\; \leq \; \min\limits_{p \in P}\{p_{i_P}-p_{j_P}\} + \min\limits_{q \in Q}\{q_{i_Q}-q_{j_Q}\} +  \min\limits_{p\in P}\{p_{j_{P}}-{p_{i_{Q}}}\} +   \min\limits_{q\in Q}\{q_{j_{Q}}-{q_{i_{P}}}\}  \label{proof:thm16102}\\
     &   \; = \; A+ D + B + E.  \notag
\end{align}
Also, by \eqref{proof:thm1613}, \eqref{proof:thm1622} and the first inequality in \eqref{proof:thm16100},  we have
\begin{align}
z &\;\leq \; \min\limits_{p \in P}\{p_{i_P}-p_{j_P}\} + \omega_{i_P}-\omega_{j_P} \label{proof:thm16103}\\
    &\; =\; \min\limits_{p \in P}\{p_{i_P}-p_{j_P}\} + \omega_{i_P}-\omega_{j_Q} + \omega_{j_Q}-\omega_{j_P} \notag \\
      &\; = \; \min\limits_{p \in P}\{p_{i_P}-p_{j_P}\} + (\omega_{j_Q}-\omega_{j_P}) + (\omega_{i_P}-\omega_{j_Q}) \notag \\
      &\; \leq \; \min\limits_{p \in P}\{p_{i_P}-p_{j_P}\} +\min\limits_{p\in P}\{p_{j_{P}}-{p_{j_{Q}}}\} + \min\limits_{q \in Q}\{q_{j_{Q}}-q_{i_{P}}\} \label{proof:thm16104}
    \; = \; A+ C+ E. 
\end{align}
Symmetrically, by \eqref{proof:thm1612}, \eqref{proof:thm1623} and the second inequality in \eqref{proof:thm16100},  we have
\begin{align}
z &\;\leq \; \min\limits_{q \in Q}\{q_{i_Q}-q_{j_Q}\} + \omega_{i_Q}-\omega_{j_Q} \label{proof:thm16105}\\
    &\; =\; \min\limits_{q \in Q}\{q_{i_Q}-q_{j_Q}\} + (\omega_{i_Q}-\omega_{j_P}) + (\omega_{j_P}-\omega_{j_Q})\notag\\
      &\; \leq \; \min\limits_{q \in Q}\{q_{i_Q}-q_{j_Q}\} +\min\limits_{p\in P}\{p_{j_{P}}-{p_{i_{Q}}}\} + \min\limits_{q \in Q}\{q_{j_{Q}}-q_{j_{P}}\} \label{proof:thm16106}
    \; = \; D+ B+ F.
\end{align}
Hence, all the values $\frac{1}{2}\left(A+B+D+E\right)$, $A+C+E$ and $B+D+F$ are  upper bounds for $z$.

If $\frac{1}{2}\left(A+B+D+E\right)\leq \min\{A+C+E, B+D+F\}$, then we can choose a feasible $\omega$ such that 
$$\omega_{i_Q}-\omega_{j_P}=-B,\; \omega_{i_P}-\omega_{j_Q}=-E,$$
$$\omega_{j_P}-\omega_{i_P}=\frac{1}{2}\left(A-B-D-E\right), \;\omega_{j_Q}-\omega_{i_Q}=\frac{1}{2}\left(D-A-B-E\right)$$
which imply $\omega_{j_Q}-\omega_{j_P}=\frac{1}{2}\left(B+D-E-A\right)$.
For this $\omega$, the equalities
in \eqref{proof:thm16100}, \eqref{proof:thm16101} and \eqref{proof:thm16102} hold, and $z$ reaches its
optimal value $\frac{1}{2}\left(A+B+D+E\right)$. 

If $A+C+E< \min\{\frac{1}{2}\left(A+B+D+E\right),\; B+D+F\}$, then we can choose a feasible $\omega$ such that 
$$\omega_{j_Q}-\omega_{j_P}=C, \;\text{and}\; \omega_{i_P}-\omega_{j_Q}=E,$$
which imply $\omega_{j_P}-\omega_{i_P}=-C-E$.
For this $\omega$, the equalities
in \eqref{proof:thm16103} and \eqref{proof:thm16104} hold, and $z$ reaches its
optimal value $A+C+E$. 
Symmetrically, 
if 
\[B+D+F<\min \{\frac{1}{2}\left(A+B+D+E\right),\; A+C+E\},\] then we can choose a feasible $\omega$ such that 
$$\omega_{j_P}-\omega_{j_Q}=F, \;\text{and}\; \omega_{i_Q}-\omega_{j_P}=B,$$
which imply $\omega_{j_Q}-\omega_{i_Q}=-B-F$.
For this $\omega$, the equalities 
in \eqref{proof:thm16105} and \eqref{proof:thm16106} hold, and $z$ reaches its
optimal value $B+D+F$.  $\Box$
\end{proof}

\subsection{Proof of Theorem \ref{thm:17}}
\begin{proof}
We only need to prove part (i) since part (ii) can be symmetrically argued.  
For any $p\in P$, by the assumptions $i(p)=i_P$ and $j(p)=j_P$, and by \eqref{equation:253}--\eqref{equation:254}, we have:
$$
\begin{array}{cccc}
    \omega_{j_P}-\omega_{i_{P}}\leq p_{i_{P}}-p_{j_P}, &  \text{and}\;\;  \omega_{l} -\omega_{j_P}\leq p_{j_P}-p_{l},\;\; \forall l\neq i_P, j_P.
  \end{array}
$$
By \eqref{equation:255}, $i_{P}\neq i_{Q}$. Note that we assume $i_{Q}\neq j_P$.
So, there exists $l\neq i_{P},j_P$ such that $l=i_{Q}$.  For any $q \in Q$, by \eqref{equation:253}--\eqref{equation:254},  we have
$$
\begin{array}{cccc}
   \omega_{j_Q}-\omega_{i_{Q}}\leq  q_{i_{Q}}-q_{j_Q}, &  \text{and}\;\;
 \omega_{l}-\omega_{j_Q}\leq    q_{j_Q} - q_{l}, \;\;\forall l\neq i_Q, j_Q.
  \end{array}
$$
By the definition of $i(p)$ and $j(p)$, we have $i_P\neq j_P$. So we have $j_Q\neq j_P$ since we assume that $i_P=j_Q$. Notice again that we assume $i_{Q}\neq j_P$.  Hence, there exists $l\neq i_{Q}, j_Q$ such that $l=j_{P}$.
So, for any $p\in P$, 

\noindent\begin{tabularx}{\textwidth}{@{}XXX@{}}
\begin{equation}\label{proof:thm1711}
\omega_{j_P}-\omega_{i_{P}}\leq p_{i_{P}}-p_{j_P},
\end{equation}&
\begin{equation}\label{proof:thm1712}
\omega_{i_Q} -\omega_{j_P}\leq p_{j_P}-p_{i_Q},
\end{equation}
\end{tabularx}
and for any $q\in Q$,

\noindent\begin{tabularx}{\textwidth}{@{}XXX@{}}
\begin{equation}\label{proof:thm1721}
\omega_{j_Q}-\omega_{i_{Q}}\leq  q_{i_{Q}}-q_{j_Q},
\end{equation}&
\begin{equation}\label{proof:thm1722}
 \omega_{j_P}-\omega_{j_Q}\leq    q_{j_Q} - q_{j_P}.
\end{equation}
\end{tabularx}
If $i_P=j_Q$, then by   \eqref{proof:thm1711} and \eqref{proof:thm1722}, we have 
\begin{align}\label{proof:thm173}
\max_{\xi\in P\cup Q} \{\xi_{j_P}-\xi_{i_P}\} \;\leq\; \omega_{i_P}-\omega_{j_P}.
\end{align}
If $i_P=j_Q$, then by adding  \eqref{proof:thm1712} and \eqref{proof:thm1721}, we have 
\begin{align}\label{proof:thm174}
\omega_{i_P}-\omega_{j_P}\;\leq\; \min\limits_{p\in P} \{p_{j_P}-p_{i_Q}\}+\min\limits_{q\in Q} \{ q_{i_{Q}}-q_{i_P}\}.
\end{align}
So, if $i_P=j_Q$, then by \eqref{proof:thm173} and \eqref{proof:thm174}, we have \eqref{eq:thm171}. 

On the other hand, if we have \eqref{eq:thm171}, then
there exist real numbers $\omega_{i_P}$ and $\omega_{j_P}$ such that
\eqref{proof:thm173} and  \eqref{proof:thm174} hold. 
By \eqref{proof:thm173}, we have \eqref{proof:thm1711} and \eqref{proof:thm1722}.
Let $\omega_{i_Q}=\omega_{i_P}+\min\limits_{p\in P}\{p_{j_{P}}-p_{i_{Q}}\}$. Then we have the 
inequality  \eqref{proof:thm1712}, and by \eqref{proof:thm174}, we have 
\begin{align*}
\max_{q\in  Q} \{q_{i_P}-q_{i_Q}\} \;\leq\; \omega_{i_Q}-\omega_{i_P}, 
\end{align*}
which is equivalent to \eqref{proof:thm1721}. 
By  \eqref{proof:thm1711}, \eqref{proof:thm1712}, \eqref{proof:thm1721} and \eqref{proof:thm1722},  the inequality \eqref{equation:253} holds, and 
the the inequality \eqref{equation:254} holds for 
$l= i_Q$ when $\xi\in P$, or for $l=j_P$ when $\xi\in Q$. 
For $l\neq i_Q$ when $\xi\in P$, or for $l\neq j_P$ when $\xi\in Q$, there always exist sufficiently small numbers for  $\omega_{l}$ such that the inequality \eqref{equation:254} holds. So the inequality \eqref{eq:thm171} guarantees the feasibility of the inequalities \eqref{equation:253} and \eqref{equation:254}. Notice that once \eqref{equation:253} and \eqref{equation:254} are feasible, there is always a non-negative number $z$ such that
\eqref{equation:252} holds. So, if we have \eqref{eq:thm171}, then the linear programming \eqref{equation:251}--\eqref{equation:254} has a feasible solution.

If a feasible solution exists, then by \eqref{equation:252},
 \begin{align}\label{proof:thm175}
    z\leq \min\limits_{p \in P}\{p_{i_P}-p_{j_P}\} + \omega_{i_P}-\omega_{j_P}, \;\; z\leq \min\limits_{q \in Q}\{q_{i_Q}-q_{j_Q}\} + \omega_{i_Q}-\omega_{j_Q}.
\end{align}
Note $i_P=j_Q$. So,  by summing up the above two inequalities and by \eqref{proof:thm1712}, 
\begin{align}
2z &\;\leq \;\min\limits_{p \in P}\{p_{i_P}-p_{j_P}\} + \min\limits_{q \in Q}\{q_{i_Q}-q_{j_Q}\} + \omega_{i_Q}-\omega_{j_P} \label{proof:thm1761} \\
    &\; \leq \; \min\limits_{p \in P}\{p_{i_P}-p_{j_P}\} + \min\limits_{q \in Q}\{q_{i_Q}-q_{j_Q}\} + \min\limits_{p\in P}\{p_{j_{P}}-{p_{i_{Q}}}\} 
    \; = \; A'+ C+ B. \label{proof:thm1762}
                                       \end{align}
                                       Also, by \eqref{proof:thm1712}, \eqref{proof:thm173} and the second inequality in \eqref{proof:thm175},  we have
                                       \begin{align}
                                                                              z &\;\leq \; \min\limits_{q \in Q}\{q_{i_Q}-q_{j_Q}\} + \omega_{i_Q}-\omega_{j_Q} \label{proof:thm1771}
\\
                                       &\; =\; \min\limits_{q \in Q}\{q_{i_Q}-q_{j_Q}\} + \omega_{i_Q}-\omega_{j_P} + \omega_{j_P}-\omega_{j_Q} \notag \\
                                       &\; = \; \min\limits_{q \in Q}\{q_{i_Q}-q_{j_Q}\} + (\omega_{i_Q}-\omega_{j_P}) + (\omega_{j_P}-\omega_{i_P}) \notag \\
                                       &\; \leq \; \min\limits_{q \in Q}\{q_{i_Q}-q_{j_Q}\} +\min\limits_{p\in P}\{p_{j_{P}}-{p_{i_{Q}}}\} + \min\limits_{\xi \in P\cup Q}\{\xi_{i_{P}}-\xi_{j_{P}}\}
                                       \; = \; C+ B+ A. \label{proof:thm1772}
                                       \end{align}
                                       Hence, both values $\frac{1}{2}\left(A'+B+C\right)$ and $A+B+C$ are  upper bounds for $z$.

If $\frac{1}{2}\left(A'+B+C\right)\leq A+B+C$, then we can choose a feasible $\omega$ such that 
                                       $$\omega_{i_P}-\omega_{j_P}=\frac{1}{2}\left(B+C-A'\right), \;\text{and}\; \omega_{i_Q}=\omega_{j_P}+B,$$
which imply $\omega_{i_Q}-\omega_{j_Q}=\frac{1}{2}\left(A'+B-C\right)$. 
For this $\omega$, the equalities
in \eqref{proof:thm175}--\eqref{proof:thm1762} hold, and $z$ reaches its
optimal value $\frac{1}{2}\left(A'+B+C\right)$. 

If $A+B+C< \frac{1}{2}\left(A'+B+C\right)$, then we can choose a feasible $\omega$ such that 
$$\omega_{i_Q}-\omega_{j_Q}=A+B, \;\text{and}\; \omega_{i_Q}=\omega_{j_P}+B,$$
which imply $\omega_{i_P}-\omega_{j_P}=-A$.
For this $\omega$, the equalities
in \eqref{proof:thm1771}--\eqref{proof:thm1772} hold, and $z$ reaches its
optimal value $A+B+C$.   $\Box$
\end{proof}

\subsection{Proof of Theorem \ref{thm:15}}
 \begin{proof}
For any $p\in P$, and for any $q\in Q$, by the assumptions $i(p) = j(q)=k_1$ and $i(q)=j(p)=k_2$, and by \eqref{equation:253}, we have:
    $$
    \begin{array}{cccc}
  \omega_{k_2}-\omega_{k_1}\leq p_{k_1}-p_{k_2},
  &  \text{and}\;\;  \omega_{k_1}-\omega_{k_2}\leq  q_{k_2}-q_{k_1}
  \end{array}
  $$
    Therefore,
  \begin{align}\label{proof:151}
  \max_{p \in P}\{p_{k_2}-p_{k_1}\} \leq \omega_{k_1}-\omega_{k_2}\leq \min\limits_{q \in Q}\{q_{k_2}-q_{k_1}\}.
  \end{align}
  So if the linear programming \eqref{equation:251}--\eqref{equation:254} has a feasible solution, then we have \eqref{eq:thm151}.

    On the other hand, if we have \eqref{eq:thm151}, then
  there exist real numbers $\omega_{k_1}$ and $\omega_{k_2}$ such that
  \eqref{proof:151}, and hence \eqref{equation:253} holds. For $l\neq k_1, k_2$, there always exist sufficiently small numbers for  $\omega_{l}$ such that the inequality \eqref{equation:254} holds. So inequalities \eqref{eq:thm151} guarantees the feasibility of the inequalities \eqref{equation:253} and \eqref{equation:254}. Notice that once \eqref{equation:253} and \eqref{equation:254} are feasible, there is always a non-negative number $z$ such that 
  \eqref{equation:252} holds. So, if we have \eqref{eq:thm151}, then the linear programming \eqref{equation:251}--\eqref{equation:254} has a feasible solution.
  
  If a feasible solution exists, then by \eqref{equation:252},
  \begin{align}\label{proof:152}
  z\leq \min\limits_{p \in P}\{p_{k_1}-p_{k_2}\} + \omega_{k_1}-\omega_{k_2}, \;\; z\leq \min\limits_{q \in Q}\{q_{k_2}-q_{k_1}\} + \omega_{k_2}-\omega_{k_1}.
  \end{align}
  By summing up the above two inequalities,
  $$2z \leq \min\limits_{p \in P}\{p_{k_1}-p_{k_2}\} + \min\limits_{q \in Q}\{q_{k_2}-q_{k_1}\}.$$
    So the value \eqref{eq:thm152} is an upper bound for $z$.
  Note that we can choose a feasible $\omega$ such that 
  $$\omega_{k_1}-\omega_{k_2}=\frac{1}{2}\left(\min\limits_{q \in Q}\{q_{k_2}-q_{k_1}\}\textcolor{black}{+}\max_{p \in P}\{p_{k_2}-p_{k_1}\} \right),$$ which satisfies the inequality \eqref{proof:151}. For this $\omega$, the equalities
  in \eqref{proof:152} hold, and $z$ reaches its
  the optimal value \eqref{eq:thm152}.  $\Box$
\end{proof}

\subsection{Proof of Theorem \ref{theorem76}}
\begin{proof}
                            For any $p\in P$, by the assumptions $i(p)=i_P$ and $j(p)=j$, and by \eqref{equation:253}--\eqref{equation:254}, we have:
                              $$
                              \begin{array}{cccc}
                            \omega_{j}-\omega_{i_{P}}\leq p_{i_{P}}-p_{j}, &  \text{and}\;\;  \omega_{l} -\omega_{j}\leq p_{j}-p_{l},\;\; \forall l\neq i_P, j.
                            \end{array}
                            $$
                              By \eqref{equation:255}, $i_{P}\neq i_{Q}$. By the definition of $i(q)$ and $j(q)$, $i(q)\neq j(q)$, and hence $i_{Q}\neq j$. 
                            So, there exists $l\neq i_{P},j$ such that $l=i_{Q}$.  Similarly, for any $q \in Q$, we have
                            $$
                              \begin{array}{cccc}
                            \omega_{j}-\omega_{i_{Q}}\leq  q_{i_{Q}}-q_{j}, &  \text{and}\;\;
                            \omega_{l}-\omega_{j}\leq    q_{j} - q_{l}, \;\;\forall l\neq i_Q, j, 
                            \end{array}
                            $$
                              and there exists $l\neq i_{Q}, j$ such that $l=i_{P}$.
                            So we have
                            \begin{align}\label{proof:1}
                            \begin{array}{cccc}
                            \forall p\in P, \;\;\omega_{j}-\omega_{i_{P}}\leq p_{i_{P}}-p_{j}, &  \text{and}\;\;  \omega_{i_Q} -\omega_{j}\leq p_{j}-p_{i_Q},
                            \end{array}
                            \end{align}
                            and
                            \begin{align}\label{proof:2}
                            \begin{array}{cccc}
                            \forall q\in Q, \;\; \omega_{j}-\omega_{i_{Q}}\leq  q_{i_{Q}}-q_{j}, &  \text{and}\;\;
                            \omega_{i_P}-\omega_{j}\leq    q_{j} - q_{i_P}.
                            \end{array}
                            \end{align}
                            Therefore,
                            \begin{align}\label{proof:3}
                            \max_{q\in Q}\{q_{i_{P}}-q_{j}\} \leq \omega_{j}-\omega_{i_{P}}\leq \min\limits_{p\in P}\{p_{i_{P}}-p_{j}\}
                            \end{align}
                            and
                            \begin{align}\label{proof:4} 
                            \max_{p\in P}\{p_{i_{Q}}-p_{j}\} \leq \omega_{j}-\omega_{i_{Q}}\leq \min\limits_{q\in Q}\{q_{i_{Q}}-q_{j}\}.
                            \end{align}
                            So, if the linear programming \eqref{equation:251}--\eqref{equation:254} has a feasible solution, then we have \eqref{eq:thm141} and \eqref{eq:thm142}.

                            On the other hand, if we have \eqref{eq:thm141} and \eqref{eq:thm142}, then
                            there exist real numbers $\omega_{j}, \omega_{i_{P}}$ and $\omega_{i_{Q}}$ such that
                            \eqref{proof:3} and \eqref{proof:4} hold, and hence \eqref{proof:1} and \eqref{proof:2} hold. For $l\neq i_{P},j, i_Q$, there always exist sufficiently small numbers for  $\omega_{l}$ such that the inequality \eqref{equation:254} holds. So inequalities \eqref{eq:thm141} and \eqref{eq:thm142} guarantee the feasibility of the inequalities \eqref{equation:253} and \eqref{equation:254}. Notice that once \eqref{equation:253} and \eqref{equation:254} are feasible, there is always a non-negative number $z$ such that 
                            \eqref{equation:252} holds. So, if we have \eqref{eq:thm141} and \eqref{eq:thm142}, then the linear programming \eqref{equation:251}--\eqref{equation:254} has a feasible solution.
                              
                              If a feasible solution exists, then by \eqref{equation:252} and the first inequality in \eqref{proof:3}, 
                            $$z \leq \min\limits_{p\in P}\{p_{i_{P}}-p_{j}\} + \omega_{i_P}-\omega_{j}\leq \min\limits_{p\in P}\{p_{i_{P}}-p_{j}\} + \min\limits_{q\in Q}\{q_{j}-q_{i_{P}}\},\;\; \text{and}$$
                              by \eqref{equation:252} and the first inequality in \eqref{proof:4}, 
                            $$z \leq \min\limits_{q\in Q}\{q_{i_{Q}}-q_{j}\} + \omega_{i_Q}-\omega_{j}
                            \leq \min\limits_{q\in Q}\{q_{i_{Q}}-q_{j}\} + \min\limits_{p\in P}\{p_{j}-p_{i_{Q}}\}.\;\;\;\;\;\;$$
                              So the maximum $z$ is given by \eqref{eq:thm143},
                            and this optimal value is reached when
                            either $\omega_{i_{P}}-\omega_{j}= \min\limits_{q\in Q}\{q_{j}-q_{i_{P}}\}$, or $\omega_{i_Q}-\omega_{j}= \min\limits_{p\in P}\{p_{j}-p_{i_{Q}}\}$. $\Box$
\end{proof}
\section{Files in the  Online Repository}\label{appB}
Table \ref{table:files} lists all files at the online repository: \url{https://github.com/HoujieWang/Tropical-SVM}

\begin{table}[h!]
\centering
\caption{Supporting Information files}
\label{table:files}
{\tiny
\begin{tabular}{||c c c||} 
 \hline
 Name & File Type & Results \\ [0.5ex] 
 \hline\hline
 \texttt{unbounded.RData} & \texttt{RData} & Remark \ref{rmk:boundedsoft2} \\
 \texttt{Algorithm1.R} & \texttt{R} & Algorithm \ref{alg:16} \\ 
 \texttt{Algorithm2.R} & \texttt{R} & Algorithm \ref{alg:17} \\ 
 \texttt{Algorithm3.R} & \texttt{R} & Algorithm \ref{alg:15} \\ 
 \texttt{Algorithm4.R} & \texttt{R} & Algorithm \ref{alg:14} \\ 
 \texttt{graph producer.R} & \texttt{R} & Figure \ref{fig:accuracy}\\
 \texttt{Genetree Data} & folder & Figure \ref{fig:accuracy}\\
 
 \texttt{Genetree Data/data\_15\%.RData} & \texttt{R Data} & Figure \ref{fig:accuracy}\\
 \texttt{Genetree Data/data\_20\%.RData} & \texttt{R Data} & Figure \ref{fig:accuracy}\\
 \texttt{Genetree Data/data\_25\%.RData} & \texttt{R Data} & Figure \ref{fig:accuracy}\\
 
 \texttt{Assignment} & folder & Figure \ref{fig:accuracy}\\
 
 \texttt{Assignment/asgn\_1\_15\%.RData} & \texttt{R Data} & Figure \ref{fig:accuracy}\\
 \texttt{Assignment/asgn\_2\_15\%.RData} & \texttt{R Data} & Figure \ref{fig:accuracy}\\
 \texttt{Assignment/asgn\_3\_15\%.RData} & \texttt{R Data} & Figure \ref{fig:accuracy}\\
 \texttt{Assignment/asgn\_4\_15\%.RData }& \texttt{R Data} & Figure \ref{fig:accuracy}\\
 
 \texttt{Assignment/asgn\_1\_20\%.RData} & \texttt{R Data}& Figure \ref{fig:accuracy}\\
 \texttt{Assignment/asgn\_2\_20\%.RData }& \texttt{R Data} & Figure \ref{fig:accuracy}\\
 \texttt{Assignment/asgn\_3\_20\%.RData }& \texttt{R Data} & Figure \ref{fig:accuracy}\\
 \texttt{Assignment/asgn\_4\_20\%.RData }& \texttt{R Data} & Figure \ref{fig:accuracy}\\
 
 \texttt{Assignment/asgn\_1\_25\%.RData} & \texttt{R Data} & Figure \ref{fig:accuracy}\\
 \texttt{Assignment/asgn\_2\_25\%.RData} & \texttt{R Data} & Figure \ref{fig:accuracy}\\
 \texttt{Assignment/asgn\_3\_25\%.RData }& \texttt{R Data} & Figure \ref{fig:accuracy}\\
 \texttt{Assignment/asgn\_4\_25\%.RData }& \texttt{R Data} &
 Figure \ref{fig:accuracy}\\

 \hline
\end{tabular}
}
\end{table}
\end{appendix}
{}

\end{document}